\documentclass[a4paper]{article}

\usepackage{authblk}
\usepackage[utf8]{inputenc}
\usepackage[english]{babel}
\usepackage[left=2.5cm, top=2cm, right=2.5cm, bottom=3cm]{geometry}
\usepackage{graphicx}
\usepackage[usenames,dvipsnames]{xcolor}
\usepackage{color}
\usepackage{import}
\usepackage{xspace}
\usepackage{enumitem}
\usepackage[hidelinks]{hyperref}
\usepackage{pgf}

\newbox{\myorcidaffilbox}
\sbox{\myorcidaffilbox}{\large\includegraphics[height=1.7ex]{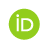}}
\newcommand{\orcid}[1]{\href{https://orcid.org/#1}{\usebox{\myorcidaffilbox}}}

\usepackage{amsthm}
\usepackage{amsmath}
\usepackage{amssymb}
\usepackage{mathtools}

\newtheorem{rem}{Remark}
\newtheorem{lemma}{Lemma}
\newtheorem{proposition}{Proposition}
\newtheorem{theorem}{Theorem}

\makeatletter
\DeclareRobustCommand\onedot{\futurelet\@let@token\@onedot}
\def\@onedot{\ifx\@let@token.\else.\null\fi\xspace}

\def\eg{{e.g}\onedot} 

\def\ie{{i.e}\onedot} 
\def\cf{{cf}\onedot}

\def\etal{{et al}\onedot}

\makeatother

\newcommand{\R}{\mathbb{R}}
\newcommand{\N}{\mathbb{N}}

\newcommand{\tr}{\mathrm{tr}\,}
\newcommand{\Id}{\mathrm{Id}}

\DeclareMathOperator*{\argmax}{arg\,max}
\DeclareMathOperator*{\argmin}{arg\,min}

\newcommand{\state}{{y}}

\newcommand{\y}{\state}

\newcommand{\vertices}{\mathcal{V}}
\newcommand{\edges}{\mathcal{E}}
\newcommand{\faces}{\mathcal{T}}
\newcommand{\edge}{e}
\newcommand{\vertex}{v}
\newcommand{\face}{t}

\newcommand{\x}{x}
\newcommand{\xref}{\hat x}
\newcommand{\numV}{{\lvert \vertices \rvert}}

\newcommand{\numF}{{\lvert\faces\rvert}}
\newcommand{\area}{{a}}
\newcommand{\mass}{M}
\newcommand{\len}{{l}}
\newcommand{\dih}{{\theta}}

\newcommand{\mat}{u}

\newcommand{\W}{\mathcal{W}}
\newcommand{\mem}{{\mbox{{\tiny mem}}}}
\newcommand{\bend}{{\mbox{{\tiny bend}}}}

\newcommand{\freeE}{\mathcal{I}}

\newcommand{\force}{{f}}

\newcommand{\dfFF}{G}

\newcommand{\dDisTen}{\mathcal{G}}

\newcommand{\lin}{{\mathrm{lin}}}
\renewcommand{\H}{H}

\def\ymin{\upsilon_{{\mbox{\tiny min}}}}
\def\ymax{\upsilon_{{\mbox{\tiny max}}}}

\begin{document}
  
\title{A Pessimistic Bilevel Stochastic Problem for Elastic Shape Optimization}

\author[1]{Johanna Burtscheidt\orcid{0000-0001-6745-0784}}
\author[1]{Matthias Claus\orcid{0000-0002-2617-0855}}
\author[2]{Sergio Conti\orcid{0000-0001-7987-9174}}
\author[3]{Martin Rumpf}
\author[3]{\\Josua Sassen\orcid{0000-0002-8069-4713}}
\author[1]{Rüdiger Schultz\orcid{0000-0003-2486-7111}}

\affil[1]{Faculty of Mathematics, University of Duisburg-Essen}
\affil[2]{Institute for Applied Mathematics, University of Bonn}
\affil[3]{Institute for Numerical Simulation, University of Bonn}

\date{}

\maketitle

\begin{abstract}
We consider pessimistic bilevel stochastic programs in which the follower maximizes over a fixed compact convex set a strictly convex quadratic function, whose Hessian depends on the leader's decision. This results in a random upper level outcome which is evaluated by a convex risk measure. Under assumptions including real analyticity of the lower-level goal function, we prove the existence of optimal solutions. We discuss an alternate model, where the leader hedges against optimal lower-level solutions, and show that solvability can be guaranteed under weaker conditions in both, a deterministic and a stochastic setting.

The approach is applied to a mechanical shape optimization problem in which the leader decides on an optimal material distribution to minimize a tracking-type cost functional, whereas the follower chooses forces from an admissible set to maximize a compliance objective.  
The material distribution is considered to be stochastically perturbed in the actual construction phase.
Computational results illustrate the bilevel optimization concept and demonstrate the interplay of follower and leader in shape design and testing.
\end{abstract}

\section{Introduction}

Bilevel programs arise from the interplay of two decision makers on different levels of a hierarchy: The leader decides first and passes the upper-level decision to the follower. Incorporating the leader's decision as a parameter, the follower then returns an optimal solution of the lower-level problem. The leader's outcome depends on both, their decision and the solution that is picked by the follower. While the first formulation of a bilevel problem dates back to a monograph on duopoly market models published in 1934 (cf. \cite{St34}), these problems have not received extensive attention until the 1970s (for more details, we refer to \cite{De02}).

In this paper, we study a class of pessimistic bilevel stochastic programs, where the lower level problem has a strictly convex quadratic objective function and a fixed feasible set. 
As an application, we study a mechanical shape optimization problem in which the leader (the designer) minimizes a tracking functional over the set of feasible material distributions, whereas the follower (the test engineer) chooses forces from an admissible set to maximize a compliance objective. 
The safety of the construction is then evaluated pessimistically with the choice of the worst possible response.
Randomness comes into play via manufacturing errors that stochastically perturb the material parameters in the actual construction phase.

In what follows, let us briefly review related work.
\emph{Bilevel programs} are nonconvex, nondifferentiable and NP-hard (non-deterministic polynomial-time-hard) \cite{Ba91}. Moreover, conceptual difficulties arise if the lower level problem has more than a single optimal solution. In this setting, one typically considers the so-called optimistic formulation, where cooperation of the follower is assumed, or takes a pessimistic stance and hedges against the worst possible outcome \cite{Le78}. It is well-known that pessimistic bilevel programs have weaker analytical properties than their optimistic counterparts.
In general, the existence of optimal solutions to a pessimistic bilevel program can only be assured under restrictive conditions including weak analyticity for the lower level objective function and strong assumptions on the structure of the lower level feasible set (cf. \cite[Theorem 4.1]{LuMiPi87}).
These difficulties can be overcome by considering a modified setting, where the leader hedges against solutions that are almost optimal for the lower level problem. 
Sufficient conditions for convergence of the modified optimal values to the original one have been established in \cite{LoMo96}. A systematic analysis of more inner regularization techniques has been recently provided by Lignola and Morgan in \cite{LiMo17,LiMo19}.

A \emph{bilevel stochastic program} arises if the problem depends on an additional random parameter, that only the follower can observe before making their decision. 
In contrast, the leader has to decide nonanticipatorily, but is aware of the underlying probability distribution. In this setting, the upper level objective function can be understood as a random variable, which allows the leader to base their decision on some statistical functional. 
The expected value is for instance considered in the very first paper on bilevel stochastic optimization \cite{PaWy99}. 
In a linear setting, more general models incorporating a variety of convex risk measures have been recently studied in \cite{BuClDe20}.
The control of a vibrating string with stochastic data has been investigated  in \cite{FaGuHeHe20} for the case of the excess-probability as a goal function.  
A level-set based approach for solving risk-averse structural topology optimization problems with random field loading and material uncertainty is given in \cite{MaHeKePe18}.

Already in 2001, Christiansen \etal \cite{ChPaWy01} studied a \emph{stochastic bilevel programming perspective in shape optimization}.  
They assume that the lower level deals with the deformation of the structure for a given shape and given forces subject to different constraints, while on the upper level the shape is decided based on an optimization of weight or a global stiffness measure. Assuming that the lower level is uniquely solvable, the authors provide sufficient conditions for the existence of optimal solutions and discuss algorithmic aspects.
Herskovits \etal \cite{HeLeDiSa00} reformulated an elastic shape optimization problem with constraints as a bilevel optimization problem.
They investigate a contact problem with non-penetration constraints on the lower level and stress constraints on the upper level. In \cite{Zu15}, Zuo investigated shape optimization of thin shells in car design as an optimistic bilevel optimization problem, where on the lower level the mass distribution along the body frame of the vehicle and on the upper level the shape of shell segments of the hull of the vehicle are optimized.
Sinha \etal \cite{SiMaDe18} recently presented a general overview on bilevel optimization also covering optimal design problems. 
In this context, they considered weight or cost optimization of a structure on the upper level and, on the lower level, the computation of displacements and stresses via minimization of the governing physical variational problem. 
To the best of our knowledge, pessimistic hierarchical optimization in shape optimization with an objective functional differing from the physical energy of the system on the upper two levels has not been investigated so far.

The approach  presented here is based on our previous work in \cite{CoHePaRuSc09,CoHePaRuSc11,CoRuScTo18}, which grew out of the aspiration to mobilize methodology from mainly economy-driven decision making under (stochastic) uncertainty in order to study PDE-constrained optimization with an emphasis on engineering-related topics such as shape optimization. The risk-neutral models and  models with risk aversion in the objective or the  constraints were treated with the classical expectation, with risk measures, or by invoking comparisons using  stochastic  dominance relations. In the spirit of this experience, the present paper is heading for models with the above bilevel features coming to the fore in the presence of uncertainty.

\medskip

The present work is organized as follows: In Section 2, we introduce a bilevel programming formulation, and in Section 3, the extension to a bilevel problem under stochastic uncertainty, which will be placed in the context of elastic shape optimization later in the paper. Based on this, we analyze both problem formulations and investigate their solvability. The application to a mechanical shape optimization problem via discrete shells is considered in Section 4 as well as its numerical optimization and the results of our numerical analysis. Finally, in Section 5, we draw conclusions and discuss possible future extensions of our work.

\section{Bilevel Problem Formulation}
Before formally introducing the bilevel problem, we briefly present the key objects. At the lowest level, $y[u, f]$ is the elastic displacement of the discrete shell which depends nonlinearly on the material parameters $u$ and linearly on the applied forces $f$.
The lower-level optimal solution set $\Psi[u]$, depending on the material parameter $u$, is the set of values of $f$ which maximize a quadratic functional in $y$.
In the upper-level of our pessimistic bilevel problem, we finally minimize the worst-case cost $J$ of the lower-level optimization with respect to the material parameters $u$.

In detail, this pessimistic bilevel problem reads as
\begin{equation}
    \label{PessimisticProblem}
    \underset{u \in \mathcal{U}}{\min} \left\{ \max_{f \in \Psi[u]} J[u,f] \right\},
\end{equation}
where $\mathcal{U} \subseteq (0,\infty)^n$ is a nonempty closed set, and $J \colon \mathcal{U} \times \R^N \to \R$ denotes the cost functional of the leader, which we assume to be continuous. In our application, this will be a tracking-type objective for a discrete shell with thickness/stiffness parameters $u$ in an admissible set $\mathcal{U}$ and applied forces $f$. Moreover, we let the lower level optimal solution set mapping $\Psi\colon \mathcal{U} \rightrightarrows \R^N$ be given by
\begin{equation}
    \label{eqdefPsiu}
    \Psi[u] \coloneqq \underset{f \in \mathcal{F}}{\argmax} \left\{ y[u, f]^\top H[u] y[u, f] \right\}
\end{equation}
with a nonempty, low-dimensional, convex and compact set of admissible forces $\mathcal{F} \subset \R^N$, a function $H\colon \R^n \to \R^{N \times N}$ such that the the restriction $H|_{\mathcal{U}}$ is continuous and takes values in the cone of symmetric positive definite matrices $\mathcal{S}^N_{++}$. 
Throughout the paper, the notation $g\colon X \rightrightarrows Y$ is used for a multifunction $g$ that maps the elements of some set $X$ to subsets of some set $Y$. The displacement $y$ depends on a vector $\mat$ of thickness/stiffness parameters and the forces $f$. 
In fact, the mapping $y\colon \mathcal{U} \times \R^N \to \R^N$ in \eqref{eqdefPsiu} is defined by the condition
\begin{equation}
    \label{LowestLevelProblem}
    \{y[u, f]\} = \argmin_{y \in \R^N} \left\{ \frac{1}{2} y^\top H[u] y - y^\top M f \right\}
\end{equation}
for some fixed matrix $M \in \mathcal{S}^N_{++}$, where uniqueness follows from $H[u] \in \mathcal{S}^N_{++}$. 
In our application, we consider a discrete shell with $n$ triangular facets subject to a force distribution $f$ in a set of admissible forces in $\R^N$, with $N$ being three times the number of vertices. 
For this case, the elastic displacement $y[u, f]$ is given as the minimizer of the total free energy of a linearized elasticity model with $H[\mat]$ denoting the Hessian of an originally nonlinear elastic energy and $M$ the mass matrix for the discrete reference shell.

\medskip

The above hierarchical problem \eqref{PessimisticProblem}-\eqref{LowestLevelProblem} can also be understood as a three-level program. However, as $H[u] \in \R^{N \times N}$ is symmetric and positive definite for any admissible material parameter $u \in \mathcal{U}$, the third-level problem in \eqref{LowestLevelProblem} is uniquely solvable. Invoking first-order optimality conditions, we obtain the explicit representation
\begin{equation}
    \label{LowestLevelSolution}
    y[u,f] =  H[u]^{-1} M f.
\end{equation}
Plugging this solution into the lower level problem yields a bilevel problem. Moreover, \eqref{LowestLevelSolution} leads to a simple expression for the lower level optimal value function $\psi\colon \mathcal{U} \to \mathbb{R}$,
\begin{equation}
    \label{lowerlevelproblem}
    \psi[u] \coloneqq \underset{f \in \mathcal{F}}{\max} \left\{ f^\top M H[u]^{-1} M f \right\},
\end{equation}
and to the reformulation of the definition of $\Psi$ in \eqref{eqdefPsiu} as
\begin{equation} 
    \label{eqPsifrompsi}
    \Psi[u] = \left\{f \in \mathcal{F} \mid f^\top M H[u]^{-1} M f = \psi[u] \right\}.
\end{equation}

\begin{lemma} \label{LemmaLowerLevelOptimalValueFunction}
    The lower level optimal value function $\psi$ defined by \eqref{lowerlevelproblem} is well-defined and continuous. In addition, the multifunction $\Psi$ is closed.
\end{lemma}

\begin{proof}
    For fixed $u$, the argument in \eqref{lowerlevelproblem} is quadratic in $f$, and in particular continuous. Since $\mathcal{F}$ is nonempty and compact, the maximum exists.
    
    For any $f\in\mathcal{F}$, the argument in \eqref{lowerlevelproblem} is a continuous
    function of $u$.  Moreover, $\psi$ is the pointwise supremum of a familiy of continuous functions and thus lower semicontinuous by \cite[Proposition~1.26(a)]{RoWe09}.
    
    We assume that $\psi$ is not upper semicontinuous, which yields $u_*\in\mathcal{U}$ and a sequence $u_j\to u_*$ such that $\lim_{j\to\infty} \psi[u_j] > \psi[u_*]$.
    For any $j$, we choose $f_j\in\mathcal{F}$ such that $\psi[u_j] = f_j^\top M H[u_j]^{-1} M f_j$. Passing to a subsequence, we can assume that $f_j$ converges to some $f_*\in\mathcal{F}$, since $\mathcal{F}$ is compact. By continuity of $f^\top M H[u]^{-1} M f$ we obtain
    \begin{equation*}
        \psi[u_*]\ge  f_*^\top M H[u_*]^{-1} M f_*=\lim_{j\to\infty}  f_j^\top M H[u_j]^{-1} M f_j=\lim_{j\to\infty} \psi[u_j],
    \end{equation*}
    a contradiction. Hence $\psi$ is continuous.
    
    \medskip
    
    Then the graph of the solution set mapping
    \begin{equation*}
        \mathrm{gph}\,\Psi = \left\{ [u,f] \in \mathcal{U} \times \mathcal{F} \mid \psi[u] -  f^\top M H[u]^{-1} M f = 0 \right\}
    \end{equation*}
    is the intersection of the closed set $\mathcal{U} \times \mathcal{F}$ and the level set of a continuous function and thus closed.
\end{proof}

\begin{proposition} \label{PropDetUSC}
    The mapping $\Phi\colon \mathcal{U} \to \R$ defined by
    \begin{equation*}
        \Phi[u] \coloneqq \max_{f \in \Psi[u]} J[u,f]
    \end{equation*}
    is well-defined and upper semicontinuous. Moreover, $\Phi$ is continuous at any $u \in \mathcal{U}$ for which $\Psi[u]$ is a singleton.
\end{proposition}

\begin{proof}
    For any fixed $u\in\mathcal{U}$, by \eqref{eqPsifrompsi} and Lemma \ref{LemmaLowerLevelOptimalValueFunction} the lower-level set solution mapping $\Psi[u]$ is a nonempty, closed subset of the compact set $\mathcal{F}$, and hence compact.
    Thus, $\Phi$ is well-defined by continuity of the upper-level cost functional $J$ on $\mathcal{U} \times \mathbb{R}^N$. Consider any sequence $\lbrace u_k \rbrace_{k \in \mathbb{N}} \subseteq \mathcal{U}$ that converges to $u \in \mathcal{U}$. By the previous considerations, there exists a sequence $\lbrace f_k \rbrace_{k \in \mathbb{N}}$ such that $[u_k,f_k] \in \mathrm{gph}\,\Psi$ and
    \begin{equation*}
        \Phi[u_k] = J[u_k,f_k]
    \end{equation*}
    holds for any $k \in \mathbb{N}$. As $\mathcal{F}$ is compact, we may assume without loss of generality that $\lbrace f_k \rbrace_{k \in \mathbb{N}}$ converges to some $f \in \mathcal{F}$. By Lemma \ref{LemmaLowerLevelOptimalValueFunction}, we have $[u,f] \in \mathrm{gph}\,\Psi$ and thus
    \begin{equation*}
        \limsup_{k \to \infty} \Phi[u_k] = \limsup_{k \to \infty} J[u_k,f_k] = J[u,f] \leq \max_{\bar{f} \in \Psi[u]} J[u,\bar{f}] = \Phi[u].
    \end{equation*}
    Hence, $\Phi$ is upper semicontinuous. If $\Psi[u]$ is a singleton, we also have
    \begin{equation*}
        \liminf_{k \to \infty} \Phi[u_k] = \liminf_{k \to \infty} J[u_k,f_k] = J[u,f] = \Phi[u],
    \end{equation*}
    which completes the proof.
\end{proof}

\begin{rem}\label{rem12}
    To understand the significance of Proposition \ref{PropDetUSC} it is useful to compute these quantities explicitly in a simple low-dimensional example. 
    Assume $n=1$, $N=2$, $\mathcal U=[\frac12,\frac32]$, $M=Id$, $H[u]=\begin{pmatrix} 1 & u-1 \\ u-1 & 1\end{pmatrix}^{-1}$, $\mathcal F=[-1,1]\times[0,1]$. Then one computes 
    $f^TMH[u]^{-1}Mf=f_1^2+f_2^2+2(u-1)f_1 f_2$. Maximizing this quantity as in \eqref{lowerlevelproblem} 
    we see that only the two points $\{\pm 1,1\}$ of $\mathcal F$ are relevant, and in particular
    $\psi[u]=2+2|u-1|$. Further, from  \eqref{eqdefPsiu} (or, equivalently,
    \eqref{eqPsifrompsi}) we obtain the set of extremal forces
    \begin{equation*}
        \Psi[u]=\begin{cases}\left\{                         
            \begin{pmatrix} 1 \\ 1 \end{pmatrix}\right\} & \text{ if } u>1,\\ \left\{                         
            \begin{pmatrix} -1 \\1 \end{pmatrix}\right\} & \text{ if } u<1,\\  \left\{                          \begin{pmatrix} 1 \\ 1 \end{pmatrix},\begin{pmatrix} -1 \\ 1 \end{pmatrix} \right\} & \text{ if } u=1.
        \end{cases}
    \end{equation*}
    Choosing for example $J[u,f]=uf_1$ one obtains
    \begin{equation*}
        \Phi[u]=\begin{cases}
            J[u,1]=u & \text{ if } u\ge 1,\\
            J[u,-1]=-u & \text{ if } u<1.
        \end{cases}
    \end{equation*}
    In particular, it is clear that on $\{u\ne 1\}$ the set-valued function $\Psi$ is a singleton and $\Phi$ is continuous, whereas $\Psi[1]$ contains two elements and $\Phi$ is not continuous, and not lower semicontinuous, at $u=1$.
\end{rem}

As this example shows, $\Phi$ arises as the objective function of a pessimistic bilevel program, where the lower level problem may have more than a single optimal solution and can thus not be expected to be lower semicontinuous in general, a fact that was already observed in  \cite[example on pages 30-31]{De02}. Note that this may prevent the bilevel program \eqref{PessimisticProblem} from having an optimal solution even if $\mathcal{U}$ is compact.

\medskip

To overcome the difficulties detailed above, we consider a model where the leader also hedges against $\eta$-optimal lower level solutions (cf. \cite{LiMo17}). Specifically, we replace $\Psi$ with the mapping $\Psi_\eta\colon \mathcal{U} \rightrightarrows \mathbb{R}^N$ defined by
\begin{equation*}
    \Psi_\eta[u] \coloneqq \left\{ f \in \mathcal{F} \mid \psi[u] - f^\top M H[u]^{-1} M f < \eta \right\}
\end{equation*}
for some positive constant $\eta$. This results in the modified upper level problem
\begin{equation}
    \label{ModifiedPessimisticProblem}
    \underset{u \in \mathcal{U}}{\min} \left\{ \sup_{f \in \Psi_\eta[u]} J[u,f] \right\}.
\end{equation}
As $\Psi[u] \subseteq \Psi_{\eta}[u]$ holds for any $\eta > 0$ and $u \in \mathcal{U}$, the optimal value in \eqref{ModifiedPessimisticProblem} yields an upper bound for the optimal value in \eqref{PessimisticProblem}.

\medskip

\begin{proposition} \label{PropPhiEta}
    The mapping $\Phi_\eta\colon \mathcal{U} \to \R$ defined by
    \begin{equation*}
        \Phi_\eta[u] \coloneqq \sup_{f \in \Psi_\eta[u]} J[u,f]
    \end{equation*}
    is well-defined and lower semicontinuous for any $\eta > 0$. In particular, \eqref{ModifiedPessimisticProblem} is solvable whenever $\mathcal{U}$ is nonempty and compact.
\end{proposition}

\begin{proof}
    First, note that $\Phi_\eta$ is well-defined and real-valued as, by continuity of $J$,
    for any $u \in \mathcal{U}$
    \begin{equation*}
        \Phi_\eta[u] \leq \max_{f \in \mathcal{F}} J[u,f] < \infty.
    \end{equation*}
    
    To prove semicontinuity, we consider a sequence $\{u_k\}_{k \in \mathbb{N}} \subseteq \mathcal{U}$ converging to some $u_*\in\mathcal{U}$. We select a sequence $\{f_j\}_{j \in \mathbb{N}} \subseteq \Psi_\eta[u_*]$, \ie $\psi(u_*) < \eta + f_j^\top H^{-1}[u_*]f_j$, such that $J[u_*,f_j]\to \Phi_\eta[u_*]$. By continuity of $\psi$ and $H$, there is 
    $K_j$ such that for all $k\ge K_j$ we have
    $\psi(u_k)<\eta+f_j^\top H^{-1}[u_k]f_j$, which is the same as
    $f_j\in \Psi_\eta[u_k]$. Therefore
    \begin{equation*}
        J[u_*,f_j]=\lim_{k\to\infty} J[u_k,f_j] \le \liminf_{k\to\infty} 
        \Phi_\eta[u_k].
    \end{equation*}
    Since $j$ was arbitrary, taking the limit $j\to\infty$ we 
    conclude
    \begin{equation*}
        \Phi_\eta[u_*]=\lim_{j\to\infty} J[u_*,f_j] \le \liminf_{k\to\infty} 
        \Phi_\eta[u_k].
    \end{equation*}
\end{proof}

\begin{rem}
    In \cite{LoMo96}, the alternate model 
    \begin{equation*}
        \underset{u \in \mathcal{U}}{\min} \left\{ \max_{f \in \bar{\Psi}_\eta[u]} J[u,f] \right\}
    \end{equation*}
    with
    \begin{equation*}
        \bar{\Psi}_\eta[u] \coloneqq \left\{ f \in \mathcal{F} \mid \psi[u] - f^\top M H[u]^{-1} M f \leq \eta \right\}
    \end{equation*}
    is considered. Under the present assumptions it can be shown that
    \begin{equation*}
        \lim_{\eta_\downarrow 0} \inf_{u \in \mathcal{U}} \left\{ \max_{f \in \bar{\Psi}_\eta[u]} J[u,f] \right\} 
        = \inf_{u \in \mathcal{U}} \left\{ \max_{f \in \bar{\Psi}[u]} J[u,f] \right\}.
    \end{equation*}
    However, the function
    \begin{equation*}
        \bar{\Phi}_\eta[u] \coloneqq \sup_{f \in \bar{\Psi}_\eta[u]} J[u,f]
    \end{equation*}
    is not lower semicontinuous in general, which is why we rather use formulation \eqref{ModifiedPessimisticProblem}.
\end{rem}

\section{Stochastic Model}
A bilevel stochastic program arises if a random vector enters the upper or lower levels as a parameter, with the information constraint that only the follower can observe the realization of the randomness before making their decision. In contrast, the leader has to decide nonanticipatorily, but is aware of the distribution of the randomness, which is independent of the leader's decision.

\medskip

In the following, we shall study a setting where the leader's decision $u$ is subject to a random perturbation. To become more specific, let $\Upsilon \colon \Omega \to \R^n$ be a random vector (i.e., a $\mathcal{B}$-Borel measurable function) on some probability space $(\Omega, \mathcal{B}, \mathbb{P})$. We obtain the following pattern of decision and observation:
\begin{equation*}
    \text{Leader decides}\ u
    \qquad\rightarrow\qquad
    \text{Realization of}\ \Upsilon
    \qquad\rightarrow\qquad
    \text{Follower decides}\ f.
\end{equation*}
In our model, the randomness results from manufacturing errors and has the following effect: Throughout \eqref{PessimisticProblem}-\eqref{LowestLevelProblem}, the leader's decision $u$ is replaced with the perturbed material vector $u \odot \upsilon$, where $\odot$ denotes the componentwise multiplication and $\upsilon$ is the realization of $\Upsilon$. In this setting, the leader seeks to ensure that the resulting material parameters are feasible regardless of the realization of the randomness. In order for the perturbed material vector to be almost-surely admissible, the leader has to choose the design parameter $u$ in the induced feasible set
\begin{equation*}
    \mathcal{U}_\Upsilon \coloneqq \left\{ u \mid u \odot \upsilon \in \mathcal{U} \quad\forall\,\upsilon \in \mathrm{supp}\,{\mu_{\Upsilon}} \right\},
\end{equation*}
where $\mu_{\Upsilon} \coloneqq \mathbb{P} \circ \Upsilon^{-1}$ is the induced Borel probability measure on $\R^n$. Note that the set $\mathcal{U}_\Upsilon$ is closed as the intersection of closed sets.
Typically, we think of a situation where $\mathrm{supp}\,\mu_{\Upsilon} \subseteq [a, b]^n$ holds for some $0 < a < b$, possibly both close to $1$.

\medskip

We will consider the stochastic extension of the classical pessimistic bilevel program \eqref{PessimisticProblem}-\eqref{LowestLevelProblem} as well as the modified version \eqref{ModifiedPessimisticProblem}. In both situations, we will assume the following assumption:
\begin{enumerate}[label=$\mathrm{(A\arabic*)}$,leftmargin=1cm, start=1]
    \item The support of $\mu_{\Upsilon}$ is bounded.\label{A1}
\end{enumerate}

In the classical setting, we will need the following additional assumptions:
\begin{enumerate}[label=$\mathrm{(A\arabic*)}$,leftmargin=1cm, start=2]
    \item $\mathcal{F}$ is a nonempty, bounded polyhedron, i.e. the convex hull of its nonempty and finite set of extreme points $\mathcal{P} \subseteq \mathcal{F}$.\label{A3}
    \item $\mu_{\Upsilon}$ is absolutely continuous with respect to the Lebesgue measure $\mathcal{L}^n$.\label{A2}
    \item There exists an open and connected set $\tilde{\mathcal{U}} \subseteq \mathbb{R}^n$, such that $\mathcal{U} \subseteq \tilde{\mathcal{U}}$, $H|_{\tilde{\mathcal{U}}}$ is real analytical,  and  it takes values in a closed subset of $\mathcal{S}^N_{++}$.\label{A4}
\end{enumerate}

From the leader's point of view, the material vector that will be passed down to the lower level after the stochastic perturbation has occurred can be understood as a random vector  $u \odot \Upsilon: \mathcal{U} \odot \Omega \to \R^n$ which is parameterized by the decision $u$. Similarly, the upper level outcome is a random variable $\Phi[u \odot \Upsilon] \in L^0(\Omega, \mathcal{B}, \mathbb{P})$ for any fixed $u$ by Proposition \ref{PropDetUSC}.  
Here and in the subsequent analysis, we denote the associated classical $L^p$-spaces with $p \in [1,\infty]$ by $L^p(\Omega, \mathcal{B}, \mathbb{P})$ and use $L^0(\Omega, \mathcal{B}, \mathbb{P})$ for the space of real-valued measurable functions.

\begin{theorem} \label{ThMainResult}
    Assume \ref{A1}-\ref{A4}, then the mapping $\mathbb{F}\colon \mathcal{U}_\Upsilon \to L^\infty(\Omega, \mathcal{B}, \mathbb{P})$ given by
    \begin{equation*}
        \mathbb{F}[u] \coloneqq \Phi\left[ u \odot \Upsilon \right]
    \end{equation*}
    is well-defined and continuous with respect to any $L^p$-norm with $p \in [1,\infty)$. 
\end{theorem}

The proof of Theorem \ref{ThMainResult} requires some preliminary work.

\begin{lemma} \label{LemmaDiscontinuities}
    Assume \ref{A3} and \ref{A4}, then the set of discontinuities of $\Phi$ is a Lebesgue null set.
\end{lemma}

\begin{proof}
    As the lower level goal function is strictly convex, we have $\Psi[u] \subseteq \mathcal{P}$ for any $u \in \mathcal{U}$. 
    By \ref{A4}, for any pair $(f,\tilde f) \in \mathcal{P} \times \mathcal{P}$ the function $G_{(f,\tilde{f})}: \tilde{\mathcal{U}} \to \mathbb{R}$ defined by
    \begin{equation*}
        G_{(f,\tilde{f})}[u] \coloneqq  f^\top M H[u]^{-1} M f - \tilde{f}^\top M H[u]^{-1} M \tilde{f}
    \end{equation*}
    is well-defined and real analytic. Consequently, the set
    \begin{equation*}
        B[f,\tilde{f}] \coloneqq \left\{ u \in \mathcal{U} \mid f, \tilde{f} \in \Psi[u] \right\} 
        \subseteq \left\{ u \in \mathcal{U} \mid G_{(f,\tilde{f})}[u] = 0 \right\}
    \end{equation*}
    of parameters for which $f$ and $\tilde{f}$ are optimal for the lower level problem is a Lebesgue null set, or we have
    \begin{equation*}
        G_{(f,\tilde{f})}[u] = 0
    \end{equation*}
    for any $u \in \mathcal{U}$ by \cite[Proposition 1]{Mi20}.
    
    \medskip
    
    Now we start from the case that $B[f,\tilde{f}]$ is a Lebesgue null set for any $f, \tilde{f} \in \mathcal{P}$ satisfying $f \neq \tilde{f}$. 
    Let $u\in\mathcal{U}$. If $\Psi[u]$ is a singleton, then 
    by Proposition~\ref{PropDetUSC}, 
    $\Phi$ is continuous at $u$. 
    Consequently, the set of discontinuity points of $\Phi$ is contained in
    \begin{equation*}
        \bigcup_{f, \tilde{f} \in \mathcal{P}, \; f \neq \tilde{f}} B[f,\tilde{f}],
    \end{equation*}
    which is a Lebesgue null set by the previous considerations.
    
    \medskip
    
    To take care of the general case, let us consider the following relation on $\mathcal{P} \times \mathcal{P}$:
    \begin{equation*}
        f \sim \tilde{f} \vcentcolon \Leftrightarrow 
        G_{(f,\tilde{f})}[u] = 0
        \quad\text{for any}\ u \in \mathcal{U}.
    \end{equation*}
    It is easy to verify that $\sim$ defines an equivalence relation and that the equivalence class of any extreme point $\tilde{f} \in \mathcal{P}$ is given by
    \begin{equation*}
        E[\tilde{f}] \coloneqq \left\{ f \in \mathcal{P} \mid G_{(f,\tilde{f})}[u] = 0 \quad\forall\,u \in \mathcal{U} \right\}.
    \end{equation*}
    By \eqref{eqPsifrompsi}, $E[\tilde f]\subseteq \Psi[u]$ if $\tilde f\in \Psi[u]\cap\mathcal{P}$. 
    Let $\tilde{\mathcal{P}} \subseteq \mathcal{P}$ contain exactly one element from each equivalence class, then $\Phi$ admits the representation
    \begin{equation*}
        \Phi[u] = \max_{\tilde{f} \in \tilde{\mathcal{P}}\cap \Psi[u]} \left\{ \max_{f \in E[\tilde{f}]} J[u,f] \right\}.
    \end{equation*}
    As $\mathcal{P}$ is finite, for any $\tilde{f} \in \tilde{\mathcal{P}}$ the mapping 
    \begin{equation*}
        u \mapsto  \max_{f \in E[\tilde{f}]} J[u,f]
    \end{equation*}
    is continuous.
    By the same argument as in the proof of Proposition~\ref{PropDetUSC},
    $\Phi$ is continuous on each set
    \begin{equation*}
        S[\tilde{f}] \coloneqq \left\{ u \in \mathcal{U} \mid \{ \tilde{f} \} = \Psi[u] \cap \tilde{\mathcal{P}} \right\}
    \end{equation*}
    of parameters for which $\tilde{f}$ is the only representative that is optimal for the lower level problem. Thus, the set of discontinuities of $\Phi$ is contained in the set
    \begin{equation*}
        N_B \coloneqq \bigcup_{f, \tilde{f} \in \tilde{\mathcal{P}}, \; f \neq \tilde{f}} B[f,\tilde{f}],
    \end{equation*}
    which is a Lebesgue null set by construction of $\tilde{\mathcal{P}}$.
    For later reference we remark that we obtained
    \begin{equation}
        \label{eqdecompositionU}
        \mathcal{U} = N_B \cup
        \bigcup_{\tilde{f} \in \tilde{\mathcal{P}}} S[\tilde f] \quad\text{with}\ \mathcal{L}^n(N_B)=0, 
    \end{equation}
    and that the sets $N_B$ and $S[\tilde f]$ for $\tilde{f} \in \tilde{\mathcal{P}}$ in the right-hand side of \eqref{eqdecompositionU} are pairwise disjoint.
\end{proof}

Throughout the subsequent analysis, we will use the notation introduced in the proof of Lemma \ref{LemmaDiscontinuities}.

\begin{proof}[Proof of Theorem \ref{ThMainResult}]
    Let $u \in \mathcal{U}_\Upsilon$.
    As any upper semicontinuous function is Borel measurable, $\mathbb{F}[u] \in L^0(\Omega, \mathcal{B}, \mathbb{P})$ follows directly from Proposition \ref{PropDetUSC}. Moreover, we have
    \begin{equation}
        \label{eqesssupfinite}
        \mathrm{ess \; sup} \, \mathbb{F}[u] \leq \max_{f \in \mathcal{F}} \sup_{\upsilon \in \mathrm{supp} \, \mu_{\Upsilon}} J[u \odot \upsilon,f] < \infty
    \end{equation}
    by continuity of $J$,  \ref{A1}  and \ref{A3}.
    
    Consider any sequence $\lbrace u_k \rbrace_{k \in \mathbb{N}} \subseteq \mathcal{U}_\Upsilon$ that converges to some $u \in \mathcal{U}_\Upsilon$. We write
    \begin{equation*}
        \lim_{k \to \infty} \left\| \mathbb{F}[u] - \mathbb{F}[u_k] \right\|_{L^p(\Omega, \mathcal{B}, \mathbb{P})}^p
        = \lim_{k \to \infty} \int_{\mathrm{supp} \, \mu_{\Upsilon}} \left|\Phi[u \odot \upsilon] - \Phi[u_k \odot \upsilon] \right|^p~\mu_{\Upsilon}(d\upsilon) .
    \end{equation*}
    The set $\{u\}\cup\{u_k \mid k\in\N\}$ is compact, so that 
    by continuity of $J$ and \eqref{eqesssupfinite} we obtain a uniform bound on the integrand. With \ref{A1} and dominated convergence we see that it suffices to prove pointwise convergence almost everywhere.
    
    Let $N_B$ be as in the proof above, and consider the set
    \begin{equation*}
        \hat N_B \coloneqq \left\{ \upsilon\in  \mathcal{U}_\Upsilon \mid u\odot \upsilon\in N_B \right\}.
    \end{equation*}
    By the change-of-variables formula we obtain $0=\mathcal{L}^n( N_B)=\prod_{i=1}^n u_i \mathcal{L}^n(\hat N_B)$ and, since $u_i>0$ for all $i$, $\mathcal{L}^n(\hat N_B)=0$.
    By \ref{A2}, $\mu_\Upsilon(\hat N_B)=0$.
    
    Fix some $\upsilon\in\mathcal{U}_\Upsilon\setminus \hat N_B$.  Then by
    \eqref{eqdecompositionU} we have
    $u\odot \upsilon\in S[\tilde f]$ for some $\tilde f\in\tilde{\mathcal{P}}$, so that in particular $\Phi$ is continuous at $u\odot \upsilon$. 
    From $u_k\to u$ with $k \to \infty$ we obtain $u_k\odot \upsilon\to u\odot \upsilon$ and therefore $\Phi[u_k \odot \upsilon] - \Phi[u \odot \upsilon]\to0$. This proves  pointwise convergence almost everywhere and concludes the proof.
\end{proof}

\begin{rem}
    The assertion of Theorem \ref{ThMainResult} does not hold for $p=\infty$.
    To see this, we consider the example of Remark \ref{rem12} and extend it to the stochastic setting taking $\Omega=[\frac9{10},\frac{11}{10}]$, $\mathbb P$ proportional to the Lebesgue measure restricted to $\Omega$, and $\Upsilon$ to be the identity, so that $\mathrm{supp}\, \mu_\Upsilon=\Omega$.
    Then $\mathbb F[u](v)=\Phi[uv]=\pm uv$, with the positive sign if and only if $uv\ge1$. For the sequence $u_k:=1-\frac1k\to1$ we have
    $\mathbb F[u_k](v)=u_kv$ for $v\ge 1/u_k$, and 
    $\mathbb F[u_k](v)=-u_kv$ for $v< 1/u_k$. 
    In particular for all $v\in (1, 1/u_k)$ we have
    $\mathbb F[u_k](v)-\mathbb F[1](v)=-u_kv-v$. Taking the 
    supremum over all such $v$ we obtain
    $\|\mathbb F[u_k]-\mathbb F[1]\|_{L^\infty(\Omega,\mathcal B,\mathbb P)}\ge 1+\frac1{u_k}\to2$, hence $\mathbb F[u_k]$ does not converge to $\mathbb F[1]$ in 
    $L^\infty(\Omega,\mathcal B,\mathbb P)$.
\end{rem}

As a generic first choice, the leader might assess the random upper level cost based on its expected value, i.e. consider the risk neutral bilevel stochastic program
\begin{equation}
    \label{eq:expected}
    \min_{u \in \mathcal{U}_\Upsilon} \left\{ \mathbb{E} \left[ \mathbb{F}[u] \right]\right\},
\end{equation}
which is well-defined by Theorem \ref{ThMainResult}. More in general, to allow for varying degrees of risk aversion, we take into account a mapping $\mathcal{R}\colon \mathcal{X} \to \mathbb{R}$ with
\begin{equation*}
    L^\infty(\Omega, \mathcal{B}, \mathbb{P}) 
    \subseteq \mathcal{X} 
    \subseteq L^0(\Omega, \mathcal{B}, \mathbb{P})
\end{equation*}
and consider the bilevel stochastic program
\begin{equation}
    \label{StochasticBilevelProgram}
    \min_{u \in \mathcal{U}_\Upsilon} \left\{ \mathcal{R}\left[ \mathbb{F}[u] \right]\right\}.
\end{equation}
$\mathcal{R}$ will typically be a \textbf{monetary risk measure} in the sense of \cite[Definition 4.1]{FoSc11} meaning it satisfies the following conditions:
\begin{itemize}
    \item Monotonicity: $\mathcal{R}[Y_1] \leq \mathcal{R}[Y_2]$ for all $Y_1, Y_2 \in \mathcal{X}$ satisfying $Y_1 \leq Y_2$ $\mathbb{P}$-almost surely.
    \item Translation equivariance: $\mathcal{R}[Y + m] = \mathcal{R}[Y] + m$ for all $Y \in \mathcal{X}$  and $m \in \R$.
\end{itemize}
Moreover, we will assume the following:
\begin{enumerate}[label=$\mathrm{(A\arabic*)}$,leftmargin=1cm,start=5]
    \item $\mathcal{R}\colon L^p(\Omega, \mathcal{B}, \mathbb{P}) \to \mathbb{R}$ with some $p \in [1,\infty)$ is convex and nondecreasing as defined above.\label{A5}
\end{enumerate}

\begin{rem}
    \ref{A5} holds for any \textbf{convex risk measure} in the sense of \cite{FrGi02} and \cite{FoSc02}, \ie for any monetary risk measure that is convex. In particular, this includes the expectation, the mean-upper semideviation of any order and the Conditional Value-at-Risk. However, as we do not assume translation equivariance, the assumption is also fulfilled for the expected excess of arbitrary order (cf. \cite[Chapter 6]{ShDeRu09}).
\end{rem}

The following result is well-known in the literature, see for example \cite[Theorem 4.1]{ChLi09}. For the convenience of the reader, we provide a short self-contained proof.

\begin{lemma}\label{lemmaconvexmonotoncont}
    Assume \ref{A5}, then the mapping $\mathcal{R}$ is continuous.
\end{lemma}

\begin{proof}
    For $f\in L^p(\Omega, \mathcal{B}, \mathbb{P})$ we denote by $|f|\in L^p(\Omega, \mathcal{B}, \mathbb{P})$ the function obtained taking the pointwise absolute value, so that $f\le |f|$, $-f\le |f|$ $\mathbb{P}$-almost everywhere.
    
    It suffices to prove that $\mathcal{R}$ is continuous in $0$, and we can assume that $\mathcal{R}(0) = 0$ 
    (otherwise we replace $\mathcal{R}$ by $\hat {\mathcal{R}}(f) \coloneqq \mathcal{R}(g_*+f)-\mathcal{R}(g_*)$).
    If $\mathcal{R}$ is not continuous, there is $\delta>0$ such that for any $j$ there is $f_j\in L^p(\Omega, \mathcal{B}, \mathbb{P})$ with $\|f_j\|_{L^p(\Omega, \mathcal{B}, \mathbb{P})}<4^{-j}$ and $|\mathcal{R}(f_j)|\ge\delta$. By convexity, $0=\mathcal{R}(0)\le \frac12 \mathcal{R}(f_j)+\frac12 \mathcal{R}(-f_j)$, which implies  $\mathcal{R}(-f_j)\ge -\mathcal{R}(f_j)$.
    By monotonicity, 
    \begin{equation*}
        \mathcal{R}(|f_j|) 
        \ge \max\left\{\mathcal{R}(f_j), \mathcal{R}(-f_j)\right\} 
        \ge \max\left\{\mathcal{R}(f_j), -\mathcal{R}(f_j)\right\}
        = \left|\mathcal{R}(f_j)\right|
        \ge \delta.
    \end{equation*}
    Let $f_* \coloneqq \sum_j 2^j |f_j|\in L^p(\Omega, \mathcal{B}, \mathbb{P})$. Using first monotonicity and then convexity,  we obtain $\mathcal{R}(f_*)\ge \mathcal{R}(2^j|f_j|) \ge 2^j \mathcal{R}(|f_j|) \ge 2^j\delta$ for any $j$, which contradicts the boundedness of $\mathcal{R}(f_*)$.
\end{proof}

\begin{theorem}
    Assume \ref{A1}-\ref{A5}, then the function $\mathcal{Q}_\mathcal{R} \colon \mathcal{U}_\Upsilon \to \mathbb{R}$ defined by 
    \begin{equation*}
        \mathcal{Q}_\mathcal{R}[u] 
        \coloneqq \mathcal{R}\left[\mathbb{F}[u] \right] 
        = \mathcal{R} \left[\Phi[u \odot \Upsilon] \right]
    \end{equation*}
    is continuous. In particular, the bilevel stochastic problem \eqref{StochasticBilevelProgram} has an optimal solution whenever the induced feasible set $\mathcal{U}_\Upsilon$ is nonempty and compact.
\end{theorem}

\begin{proof}
    As $\mathcal{R}$ is continuous by Lemma \ref{lemmaconvexmonotoncont}, the result follows from Theorem \ref{ThMainResult}.
\end{proof}

Let us now consider the stochastic version of the modified problem \eqref{ModifiedPessimisticProblem}, where the leader hedges against all $\eta$-optimal lower level solutions. For this, we will use the notion of law-invariant risk measure: 
\begin{equation*}
    \mathcal{R}[Y_1] 
    = \mathcal{R}[Y_2]\quad \text{for all}\ Y_1, Y_2 \in \mathcal{X}\ \text{with}\ \mathbb{P} \circ Y_1^{-1} 
    = \mathbb{P} \circ Y_2^{-1},
\end{equation*}
\ie for all $Y_1$, $Y_2$ which induce the same Borel probability measure.
The following existence result is obtained for law-invariant, convex risk measures under weaker assumptions, where we no longer restrict the analysis to polyhedral $\mathcal{F}$ and real analytic $H$.

\begin{theorem}
    Assume \ref{A1} and \ref{A5} and let $\mathcal{R}$ be translation equivariant as well as law-invariant. Then the mapping $\mathcal{Q}_{\mathcal{R}, \eta}\colon \mathcal{U}_\Upsilon \to \mathbb{R}$ given by
    \begin{equation*}
        \mathcal{Q}_{\mathcal{R}, \eta}[u] 
        \coloneqq \mathcal{R} \left[\Phi_\eta \left[ u \odot \Upsilon \right] \right]
    \end{equation*}
    is well-defined and lower semicontinuous. In particular, the bilevel stochastic program
    \begin{equation*}
        \min_{u \in \mathcal{U}_\Upsilon} \left\{ \mathcal{Q}_{\mathcal{R}, \eta}[u] \right\}
    \end{equation*}
    is solvable, whenever $\mathcal{U}_\Upsilon$ is nonempty and compact.
\end{theorem}

\begin{proof}
    First, note that $\Phi_\eta \left[u \odot \Upsilon \right] \in L^0(\Omega, \mathcal{B}, \mathbb{P})$ and
    \begin{equation*}
        \left\| \Phi_\eta \left[u \odot \Upsilon \right] \right\|_{L^\infty(\Omega, \mathcal{B}, \mathbb{P})} 
        \leq \sup_{\upsilon \in \mathrm{supp} \, \mu_\Upsilon} \Phi_\eta \left[u \odot \upsilon \right] 
        \leq \sup_{\upsilon \in \mathrm{supp} \, \mu_\Upsilon} \sup_{f\in \mathcal{F}} J[u \odot \upsilon,f] < \infty
    \end{equation*}
    hold for any $u \in \mathcal{U}_\Upsilon$ by Proposition \ref{PropPhiEta} and \ref{A1}. Thus, $\mathcal{Q}_{\mathcal{R}, \eta}$ is well-defined.
    
    \medskip
    
    Let $\mathcal{R}^\ast$ denote the convex conjugate of $\mathcal{R}$ (cf. \cite[Theorem 2.1]{KaRu09}), then $\mathcal{R}$ admits a robust representation as
    \begin{equation*}
        \mathcal{R}[\mathcal{Y}] 
        = \sup_{\mathbb{P}' \in \mathrm{Env}} \left\{ \mathbb{E}_{\mathbb{P}'}[\mathcal{Y}] - \mathcal{R}^\ast[\mathbb{P}'] \right\} \quad\forall\,\mathcal{Y} \in L^p(\Omega, \mathcal{B}, \mathbb{P}),
    \end{equation*}
    where the risk envelope $\mathrm{Env}$ is a subset of the normed positive part of the dual space of $L^p(\Omega,\mathcal{B},\mathbb{P})$ by \cite[Corollary 2.3, Theorem 2.4]{KaRu09}. Fix any $\mathbb{P}' \in \mathrm{Env}$. We shall show that the mapping $u \mapsto \mathbb{E}_{\mathbb{P}'} \left[ \Phi_\eta \left[ u \odot \Upsilon \right]\right]$ is lower semicontinuous. The result then follows because the pointwise supremum of lower semicontinuous functions is lower semicontinuous (cf. Lemma~\ref{LemmaLowerLevelOptimalValueFunction} and \cite[Proposition 1.26 (a)]{RoWe09}).
    
    \medskip
    
    Consider any sequence $\lbrace u_k \rbrace_{k \in \mathbb{N}} \subseteq \mathcal{U}_\Upsilon$ that converges to some $u \in \mathcal{U}_\Upsilon$. Without loss of generality, we assume that $u_k \in B_1(u) \cap \mathcal{U}_\Upsilon$ holds for any $k \in \mathbb{N}$, where $B_1(u)$ denotes the open Euclidean unit ball around $u$ (We denote its closure by $\overline{B}_1(u)$). By definition,
    \begin{equation*}
        \Phi_\eta \left[ u_k \odot \Upsilon \right] 
        \geq \min_{\upsilon \in \mathrm{supp} \mu_\Upsilon} \min_{u \in \overline{B}_1(u) \cap \mathcal{U}_\Upsilon} \min_{f \in \mathcal{F}} J \left[u \odot \upsilon,f \right] 
        =: \underline{J}
    \end{equation*}
    holds for any $k \in \mathbb{N}$. As $J$ is continuous and $\mathrm{supp} \, \mu_\Upsilon$, $\overline{B}_1(u) \cap \mathcal{U}_\Upsilon$, and $\mathcal{F}$ are nonempty and compact, we have $\underline{J} \in \mathbb{R}$. Thus, Fatou's Lemma yields
    \begin{equation*}
        \liminf_{k \to \infty} \mathbb{E}_{\mathbb{P}'}\left[ \Phi_\eta \left[ u_k \odot \Upsilon \right]\right] 
        \geq \int_{\Omega} \liminf_{k \to \infty} \Phi_\eta \left[ u_k \odot \Upsilon \right]~\mathbb{P}'(d\omega) 
        \geq \mathbb{E}_{\mathbb{P}'}\left[ \Phi_\eta \left[u \odot \Upsilon \right]\right],
    \end{equation*}
    which completes the proof.
\end{proof}

\section{Application: Discrete Shells}

In this section, we will apply bilevel optimization to a mechanical shape optimization problem.
Our aim is to determine the optimal elastic design of curved roof-type constructions.
The leader in this setup is the construction engineer who aims at minimizing a tracking-type functional via optimizing the distribution of material on a prescribed roof geometry.
Due to production errors, the material distribution is considered to be stochastically perturbed in the actual construction phase.
The follower is a test engineer, who is performing a worst-case analysis and considers within a given set of possible forces---for example wind and roof load---those that maximize the compliance functional. 

\subsection{General Setting and Problem Formulation}
Our model problem is taken from the literature on geometric design \cite{VoHoWa12}, but our mechanical perspective is not self-supporting 
structures but instead architectural structures composed of discrete thin shells. Indeed, 
we model the mechanical properties of a roof construction using an adaptation of the discrete elastic shell model by Grinspun \etal \cite{GrHiDeSc03}, in which the geometry is a triangular surface and each triangle is considered as a construction panel, 
with joints at the edges.
The membrane distortion deforms the individual panels, whereas the bending distortion leads to a change of the dihedral angle between pairs of panels that share an edge.
Let us emphasize that the discrete shell approach is a design tool and does not act as a computational tool for the full elastostatic modeling in a later planning stage.
In fact, we consider the discrete shell model mainly as a testbed for the proposed bilevel optimization approach. 
We underline this by reporting all physical quantities without units.

Comparing with the notation in the previous section, the design parameter $u$ will represent the thickness of the shell, $f$ the applied forces, and $y$ the resulting displacement of the shell. The minimization in equation \eqref{LowestLevelProblem}
then corresponds to the solution of a linear elasticity problem in \eqref{eq:free_energy}, with $H[u]$ representing the elastic energy. The problem in \eqref{eqdefPsiu} corresponds to the follower optimizing compliance. The leader's cost functional $J$ in \eqref{PessimisticProblem} measures the deviation from the prescribed shape and is defined in \eqref{eqdefJmaff} below.

We consider the simplicial mesh of a discrete shell  $\mathcal{S}_h = (\vertices, \edges, \faces)$ consisting of sets of vertices 
$\vertices$, edges $\edges \subset \vertices \times \vertices$ and triangular faces $\faces \subset \vertices \times \vertices \times \vertices$. 
In what follows, we use maps defined on the different elements of such a mesh instead of vectors used in the theoretical considerations above.
For example, a map $w \colon \vertices \to \R^k$ assigning each vertex a value in this section corresponds to a vector $\R^{k\numV}$ from the previous sections and similarly for functions defined on edges and faces.
We denote evaluations $w(\vertex)$ of such a map also via indexing to simplify notation, \ie $w_\vertex \coloneqq w(\vertex) \in \R^k$.

The geometry of a discrete shell is given by a map \(\x \colon \vertices \to \R^3\) subject to the constraint that for each face there is no straight line in $\R^3$ containing all three vertices, \ie no triangle degenerates to a line. 
Thus, each triangle $\face \in \faces$ with vertices $\vertex_0$, $\vertex_1$, $\vertex_2$ can be parametrized over the reference triangle in $\R^2$ with vertices $(0,0)$, $(1,0)$ and $(0,1)$ via the affine map $\x_\face$ interpolating $\x(\vertex_0)$, $\x(\vertex_1)$, $\x(\vertex_2)$. We denote by $D\x_\face$ the differential of this affine map for face $\face$, so that the associated metric tensor in the same face is
\begin{equation*} 
    \dfFF[\x_\face] \coloneqq ( D\x_\face)^\top D\x_\face.
\end{equation*}
We denote by $\xref\colon\vertices\to\R^3$ the fixed stress-free reference configuration of the discrete shell, and parametrize the deformed configuration \(\x = \xref+\y\) in terms of the elastic displacement of the vertices \(\y \colon \vertices \to \R^3\).
We denote by \(\len_\edge\) the length of an edge $\edge\in\edges$ and by \(\area_\face\) the area of a face $\face\in\faces$ in the reference configuration. 
Then,  \(\area_\edge \coloneqq \tfrac{1}{3}(\area_\face + \area_{\face'})\) is a corresponding edge-associated area, where  \(\face\) and \(\face'\) are the two faces adjacent to the interior edge \(\edge \in \edges\); correspondingly \(\area_\vertex \coloneqq \frac{1}{3} \sum_{\face \in \faces_\vertex} \area_\face\) a vertex-associated area for the ring of faces \(\faces_\vertex\) around a vertex \(\vertex \in \vertices\). 

The design variable is the material thickness parameter, which is assumed to be constant on each of the triangles and is denoted by \(\mat\colon \faces \to (0,\infty)\).
In order to evaluate the bending contribution to the energy, 
see \eqref{eq:bending_energy} below, we shall use on an interior edge $\edge$ the averaged thickness  \(\mat_\edge \coloneqq \tfrac12(\mat_\face+ \mat_{\face'})\) of the two triangles \(\face\) and \(\face'\) sharing the edge \(\edge\).

\paragraph{Variational Formulation of Discrete Shells.}

In the modeling of thin shells, the elastic stored energy is typically the sum of two terms: the stored energy caused by in-plane membrane distortion and the stored energy reflecting bending distortion \cite{ciarlet2000Shells,loveElasticity}. 
The two terms scale linearly and cubically, respectively, in the thickness of the shell.

For a displacement $\y$, the Cauchy-Green strain tensor measuring the change of lengths, and consequently area, of a face $\face$ is given by 
\begin{equation*}
    \dDisTen[\y] \coloneqq \left( \dfFF[\xref_\face] \right)^{-1} \dfFF[(\xref+\y)_\face].
\end{equation*}
Then, the membrane energy depends on this tensor and is defined as 
\begin{equation*}
    \W_\mem[\mat,\y] \coloneqq \sum_{\face \in \faces}  \area_\face \, \mat_\face \, W_\mem(\dDisTen[\y]\vert_\face),
\end{equation*}
where we use the neo-Hookean energy density 
\begin{equation*}
    W_{\mem}(A) \coloneqq \frac{\mu}{2}\tr A + \frac{\lambda}{4}\det A -\left(\mu+\frac{\lambda}{2}\right)\log \det A - \mu - \frac{\lambda}{4}.
\end{equation*}
The linearization of this energy coincides with the planar, isotropic, linearized elasticity model with Lam\'e-Navier coefficients \(\mu\) and \(\lambda\) \cite{ciarlet19883d,loveElasticity}.
In the following, we use \(\mu = \lambda = 1\).

For the bending energy, we follow \cite{HeRuSc14} and use an adaptation of the discrete shell bending energy introduced in \cite{GrHiDeSc03}.
It measures the change of the dihedral angles between a pair of neighboring triangles $\face$ and $\face'$ due to the displacement $\y$ in the configuration $\x$. The angle is computed as
\(\dih_\edge(\x) \coloneqq \arccos ( n_{\face}(\x)^\top n_{\face'}(\x) )\), where  $n_{\face}(\x)$ and $n_{\face'}(\x)$ are the unit normals generated by the deformation $x$, and the energy takes the form
\begin{equation}
    \label{eq:bending_energy}
    \W_\bend[\mat,\y] \coloneqq \gamma  \sum_{\edge\in\edges} \mat_\edge^3 \cdot \frac{(\dih_\edge(\xref+\y) - \dih_e(\xref))^2}{\area_\edge}\len_\edge^2
\end{equation}
for some constant $\gamma>0$, which in continuum models can be expressed in terms of $\lambda$ and $\mu$.  We use $\gamma=1$.

The stored elastic energy $\W[\mat,\y]$ is the sum of these two energies,
\begin{equation*}
    \W[\mat,\y] \coloneqq  \W_\mem[\mat,\y] + \W_\bend[\mat,\y],
\end{equation*}
so that the total free energy in the presence of external forces $\force \colon \vertices \to \R^{3}$  reads as
\begin{equation}
    \label{eq:free_energy}
    \freeE[\mat, \force,\y] = \W[\mat,\y] - \force^\top \mass \y, 
\end{equation}
where $\mass$ is a diagonal mass matrix in $\R^{3\numV \times 3\numV}$ with entries $\area_\vertex$ at positions $(i,i)$ with $i=3j-k$ for $j=1,\ldots, \numV$ and $k=0,1,2$.
The elastic displacements resulting from applying the forces to the reference configuration are the minimizers of this energy.

\medskip

In what follows, we restrict ourselves to the linearization of this model. 
We denote by $\H[\mat] \coloneqq \partial^2_{\y\y} \W[\mat,0]$ the Hessian of the stored elastic energy, 
and obtain the linearized stored elastic energy
\begin{equation*}
    \W^\lin[\mat,\y] \coloneqq \frac{1}{2} \y^\top \H[\mat] \y
\end{equation*}
as well as the linearized total free energy
\begin{equation*}
    \freeE^\lin[\mat, \force, \y] \coloneqq \W^\lin[\mat,\y] - \force^\top \mass   \y,
\end{equation*}
whose minimization corresponds to the innermost problem introduced in \eqref{LowestLevelProblem}.
Prescribing suitable boundary data $\y_\vertex=0$ on a set of at least three vertices $\vertex \in \vertices$, which do not lie on a line, one can deduce (\cf \cite{HeRuSc14}) that $\H[\mat]$ is a positive-definite matrix.
As written above expression \eqref{LowestLevelSolution}, for every $\mat$ and $\force$ the energy 
$\freeE^\lin[\mat, \force, \cdot]$ has a unique minimizer, which is also the unique solution of the associated Euler-Lagrange equation
\begin{equation*}
    0 = \partial_\y \freeE^\lin[\mat, \force, \y] =  \H[\mat] \y - \mass \force.
\end{equation*}

\paragraph{The Optimization Problem.}

To complete our practical optimization problem, we need to specify the admissible set of material parameters \(\mathcal{U}\), the admissible set of force parameters \(\mathcal{F}\), and the cost functional of the leader \(J\).
The objective of the lower level optimal value function \(\psi\) is already completely defined in \eqref{lowerlevelproblem} and equals the compliance functional evaluated for the displacement $\y[\mat, \force]$, \ie 
\begin{equation*}
    \psi[u] = \underset{f \in \mathcal{F}}{\max} \left\{ \force^\top\mass  \y[\mat, \force] \right\} = \underset{f \in \mathcal{F}}{\max} \left\{\force^\top\mass H[u]^{-1} \mass \force \right\}.
\end{equation*}

The admissible set of force parameters \(\mathcal{F}\) is assumed to consist of linear combinations of a small number of different load scenarios.
We assume that the forces are of the type \(f = B F\), where $F\in \R^d$ for some \(d \ll 3\numV\) are the coefficients,  and the columns \(B_j\) of the matrix \(B\in\R^{3\numV \times d}\) are the basis of the $d$-dimensional subspace of forces.
Therefore, each \(B_j\in  \R^{3\numV}\) represents a force distribution on the reference configuration $\hat \x$ which is then scaled with $F_j\in\R$, for $j=1,\dots, d$.
The components of these basis vectors could be determined, for example,  from the
location of the vertex or the inclination of the triangular faces sharing a vertex.
Furthermore, we consider different constraints on the values of the scale factors $F_j$, \ie we assume that the set $\mathcal{F}$ is given by $\bigcap_{k=1}^K \mathcal{F}_k$ with 
\begin{equation*}
    \mathcal{F}_k \coloneqq \left\{ BF\in \R^{3|\vertices|} \mid F\in \R^d,\, \mathcal{Q}^F_k(F) \geq 0 \right\}
\end{equation*}
for some smooth functions $\mathcal{Q}^F_k$ for $k=1,\ldots, K$.
For example, if $\mathcal{F}$ consists of the forces which fulfill $\lvert F \rvert \leq \mu$ then one might choose 
$d=3|\vertices|$, $B=\Id$, $K=1$, and 
$\mathcal{Q}_1^F(F) = \mu^2 - \lvert F \rvert^2$.

In the problem of the leader, we constrain the material thickness parameter $\mat$ elementwise from below and from above, 
and we assume that the total volume of material, determined via the discrete integral of $\mat$, is below some fixed positive parameter. 

Lastly, the upper level cost functional is considered to be of tracking-type and measures the  squared discrete  $L^2$-norm of the 
displacement on a predefined tracking subset of the whole shell,
\begin{equation}
    \label{eqdefJmaff}
    J[\mat,f] \coloneqq \y[\mat,f]^\top\chi  \odot M \y[\mat,f] 
    =\sum_{\vertex\in\vertices} \chi_\vertex M_{\vertex\vertex} 
    \left|\y_\vertex[\mat,f] \right|^2.
\end{equation}
Here $\chi \colon \vertices \to \{0,1\}$ is a discrete characteristic function with value $1$ at vertices in the tracking set and $0$ elsewhere.

\medskip

In the stochastic setting, we restrict ourselves to the expected value $\mathbb{E}\left[ \mathbb{F}[u] \right]$ as the risk measure  for the optimization (cf. \eqref{eq:expected}).
Furthermore, the stochastic perturbation of the distribution of the thickness parameter $\mat$ is given by i.i.d.\ normal distributions for each parameter, \ie we consider the perturbed material \(\mat \odot \Upsilon\) for \(\Upsilon \sim \mathcal{TN}(1,\sigma^2,\ymin, \ymax)^\numF\), 
where \(\mathcal{TN}(1,\sigma^2,\ymin, \ymax)\) is the truncated normal distribution with average $1$ and standard deviation $\sigma$, truncated to the interval $[\ymin,\ymax]$. In practice, we take  $\sigma\le 0.2$,
$\ymin=10^{-2}$ and $\ymax=2$, so that the truncation has little effect and $\sigma$ is almost identical to the standard deviation of $\Upsilon$.

We further fix constants $0<\mat^-<\mat^+$ and $V^+>0$ and define implicitly $\mathcal{U}$ by the condition
\begin{equation*}
    \mathcal{U}_\Upsilon = \left\{ \mat \colon \faces \to \R \mid \mat^- \leq \mat_\face \leq \mat^+ \quad\forall\, \face \in \faces, \, \sum_{\face \in \faces} \area_\face \mat_\face \leq V^+ \right\} \subset (0,\infty)^\numF.
\end{equation*}

%fffffffffffffffffffffffffffffffffffffffff
\begin{figure}[t]
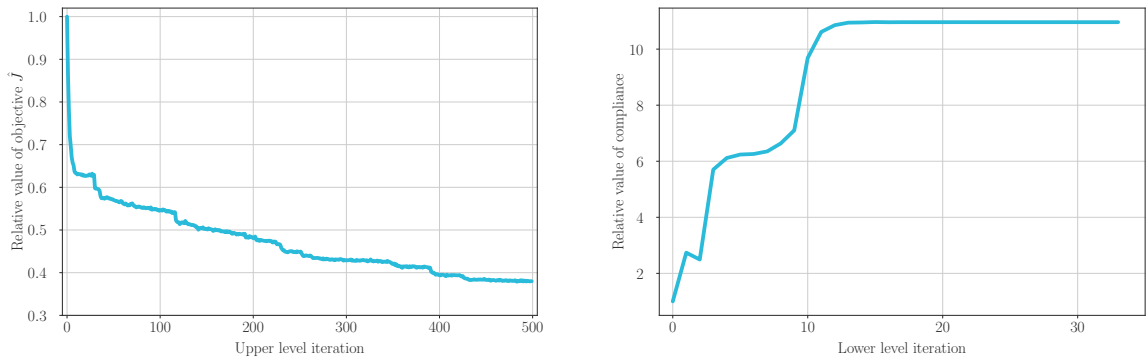

	\centering
	\hfil
	\resizebox{0.45\textwidth}{!}{\import{figures/convergence/}{Leader.pgf}} \hfil\hfil
	\resizebox{0.45\textwidth}{!}{\import{figures/convergence/}{Follower.pgf}}\hfil
	\caption{
		Left: upper level relative cost values $\hat{J}[\mat^i]/\hat{J}[\mat^0]$ for the iterates of the stochastic gradient descent method 
		in the example shown in the bottom row of Figure~\ref{fig:toymodel}.
		Right: corresponding lower level compliance cost 
		$\frac{\y[\mat, BF^j]^\top \H[\mat]\, \y[\mat, BF^j]}{\y[\mat, BF^0]^\top \H[\mat]\, \y[\mat, BF^0]}$ for iterates of the Newton-type method 
		for the follower problem in the first upper level descent step and the initial material distribution.} 
	\label{fig:convergence}
\end{figure}
%fffffffffffffffffffffffffffffffffffffffff
\subsection{Numerical Optimization}

To numerically solve the bilevel problem \eqref{PessimisticProblem} in the presented setting, it is convenient to replace the restriction of $u$ and $f$ to admissible sets $\mathcal{U}_\Upsilon$ and $\mathcal{F}$ by smooth approximations and then to deal with a differentiable problem. 
In our implementation, we achieve this by using logarithmic barrier functions, as commonly used in interior point methods (see \eg textbook \cite{NoWr06}). Hence, with the structural assumptions on the set of admissible forces introduced above, we define the smoothed follower problem by
\begin{equation}
    \label{eqpsialphanumeric}
    \Psi_\alpha[\mat] \coloneqq \underset{F \in \R^{d}}{\argmax} \left\{ \y[\mat, BF]^\top \H[\mat]\, \y[\mat, BF] + \alpha^F \sum_{k=1}^K
    \log \left(\mathcal{Q}^F_k(F) \right) \right\},
\end{equation}
where $\alpha^F > 0$ is an appropriate scaling factor for the barrier terms.

To compute the minimizers in \eqref{eqpsialphanumeric}, we do not aim at a global minimization approach but rather use an ascent method (see below) to compute isolated local minimizers.
Thus, we assume in the numerical optimization of the leader problem, that the solution of the follower problem is of such type. 
This allows us to apply conventional nonlinear optimization algorithms. 
In this framework, the maximizer and the set $\Psi_\alpha$ be interchangeable.
In the examples considered below, this assumption is justified by the use of asymmetric triangulations, and additionally by the symmetry-breaking random perturbations of the material thickness.
Thus, the logarithmic barrier formulation of the expected value optimization problem for the leader is
\begin{equation*}
    \underset{\mat \in \R^{|\faces|}}{\min} \left\{ \mathbb{E} \left[ J \left[\mat\odot \Upsilon,\Psi_\alpha[\mat\odot \Upsilon] \right] \right]  
    - \alpha^u \sum_{\face \in \faces} \area_\face \left( \log(\mat_\face-\mat^-) + \log(\mat^+ - \mat_\face)\right)
    - \alpha^V \log\left(V^+ - \sum_{\face \in \faces} \area_\face \mat_\face \right) \right\}
\end{equation*}
for scaling factors $\alpha^u, \,\alpha^V >0$  as before.

\medskip

This regular reformulation of the optimization problem can be solved numerically using a stochastic gradient method. 
For PDE-constrained  shape  optimization  problems  under  uncertainty, this method is analyzed in \cite{GeLoWe21}.
In our case, the smoothed follower problem is a deterministic and smooth optimization problem, and computing its first and second derivatives is straightforward.
Thus, we use a Newton-type method with Armijo backtracking line search (\cf \cite[Algorithm 3.2]{NoWr06}) to compute its optimizers.
The gradients of the smoothed bilevel problem can be computed via the general procedure of shape optimization calculus and thus, we employ stochastic gradient descent \cite{RoMo51} to solve it.
To this end, in each iteration of the descent algorithm, we draw finitely many samples \(\upsilon^1,\ldots, \upsilon^K\) from the distribution of the material perturbation.
In the experiments, we always chose $K=128$.
Using these samples, we approximate the expected value by the empirical risk \(\hat{J}[\mat] \coloneqq \tfrac{1}{K}\sum_{k=1}^K J \left[\mat\odot \upsilon^k,\Psi_\alpha[\mat\odot \upsilon^k] \right]\).
Then a new iterate is computed by taking a step in the direction of the negative gradient of the combination of the empirical risk and the logarithmic barrier terms.
Figure~\ref{fig:convergence} depicts the decrease of the upper level cost functional over the iterations of the stochastic descent algorithm and the increase of the lower level compliance cost when solving the follower problem for the initial material distribution.
Latter solves of the follower problem typically require 10 to 30 iterations of the Newton-type method per outer iteration.

We have implemented our method in C\texttt{++} with the Geometric Optimization And Simulation Toolbox (GOAST) \cite{GOAST}, where we use the Eigen library \cite{GuJa10} for numerical linear algebra and CHOLMOD \cite{ChDaHa08} from the SuiteSparse collection as direct linear solver. 
The code is available under \url{https://gitlab.com/numod/bilevel-shape-optimization}.

\subsection{Numerical Results}
We applied the bilevel shape optimization method in a proof-of-concept study of discrete shells representing curved roofs. We fix an orientation so that the negative $Z$-axis is in the direction of gravity and
the supporting ground is in the $XY$-plane. 
For each geometry, we fix a set of Dirichlet vertices near the ground plane, representing the points on which the structure is supported, and also fix the material thickness of the corresponding triangles.
This removes these variables from the optimization.

The construction is exposed to two types of forces.
First, there are forces emulating wind hitting the structure. 
For a given wind direction and strength, the force on each part of the roof depends on the local orientation. 
We assume that the magnitude of the force on a vertex is proportional to the absolute value of the scalar product between the vertex normal (given as the average of the normals of the triangles adjacent to the vertex) and the  wind direction.
For simplicity, we only consider a two-dimensional subset of possible forces, spanned by the basis vectors \(B_1\) and \(B_2\) which represent wind along the positive $X$- and $Y$-axis, respectively. 
The direction and magnitude of the wind are then controlled by the scale factors $F_1$ and $F_2$.
We fix a maximal magnitude of wind-type force $F_{\mathrm{max},xy}$ and use the constraint function \(\mathcal{Q}_1^F(F) \coloneqq F^2_{\mathrm{max},xy} - \left(F_1^2+ F_2^2\right) \) in \eqref{eqpsialphanumeric}.
An example of these two basis vectors demonstrating the dependence on the orientation of the normal is shown in the second and third panels of Figure \ref{fig:setup}.
%fffffffffffffffffffffffffffffffffffffffff
\begin{figure}[t]
	\centering
	\includegraphics[width=0.24\textwidth]{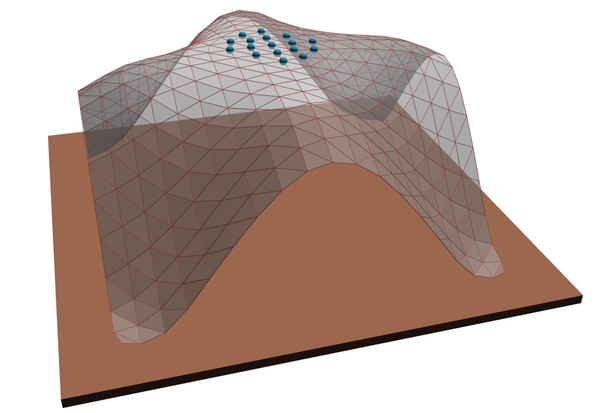}
	\includegraphics[width=0.24\textwidth]{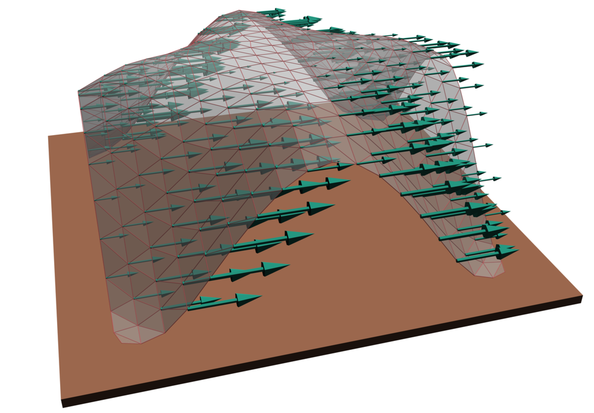}
	\includegraphics[width=0.24\textwidth]{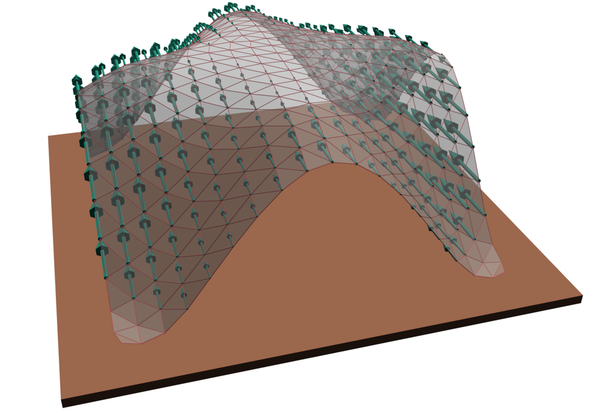}
	\includegraphics[width=0.24\textwidth]{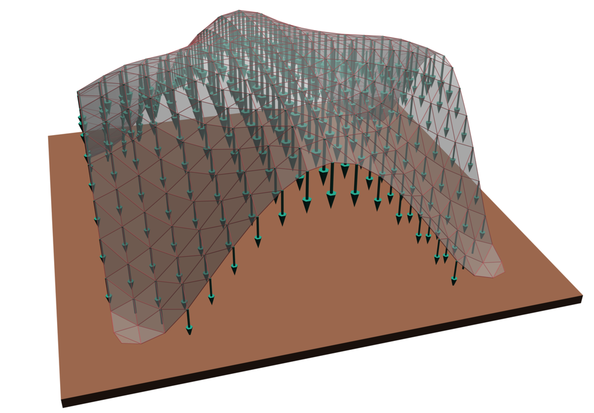}
	\caption{The first panel shows the geometry of the roof structure, with the tracking set on the roof plateau marked with dots. The Dirichlet nodes are the vertices on the horizontal plane at the corners. The other three panels show the three basis force fields $B_1$ (horizontal wind in the $X$ direction), $B_2$ (horizontal wind in the  $Y$ direction) and $B_3$ (vertical gravitational force caused by an overlay on the roof). The scale of the force arrows is arbitrary.
	} \label{fig:setup}
\end{figure}
%fffffffffffffffffffffffffffffffffffffffff
Second, we consider a vertical force,  which could emulate the weight of snow or water overlay on the roof.  
The magnitude of the corresponding basis vector \(B_3\) on each vertex is the absolute value of the scalar product between 
the vertex normal and the $Z$-axis and is shown in Figure \ref{fig:setup} on the far right.
The magnitude of gravitational load is controlled by the scale factor $F_3$, we ensure that it is pointing downward via \(\mathcal{Q}_2^F(F) \coloneqq F_3 \) and limit its magnitude via \(\mathcal{Q}_3^F(F) \coloneqq F_{\mathrm{max},z} - F_3\), where $F_{\mathrm{max},z}$ is the maximal magnitude of the gravitational force.
Therefore the admissible set \(\mathcal{F}\) is a cylinder with radius $F_{\mathrm{max},xy}$ and height $F_{\mathrm{max},z}$.

We performed most of our investigations on the simple roof geometry shown in Figure~\ref{fig:setup}.
For this problem, the basic parameters, which are used in the examples if not indicated otherwise,  are as follows.
The roof geometry is almost filling a box of \(20  \times 20 \times 10\), the maximal horizontal load is
\(F_{\mathrm{max},xy} = 0.0015\) and the vertical one \(F_{\mathrm{max},z} = 2 F_{\mathrm{max},xy}\).
The elementwise bounds on the material thickness are
\(\mat^- = 0.01\) and \(\mat^+ = 0.2\). 
The volume of the material is bounded by \(V^+ = 60\) and the strength of the stochastic variation is fixed by \(\sigma = 0.1\).
The weights of the barrier terms were \(\alpha^F = 10^{-4}\), \(\alpha^u = 1\), and \(\alpha^V = 10^{-5}\).
For the leader, we consider a tracking set restricted to the central region of the roof plateau as shown in the first panel of Figure~\ref{fig:setup}. 
%fffffffffffffffffffffffffffffffffffffffff
\begin{figure}[t]
	\centering
	\includegraphics[width=0.45\textwidth]{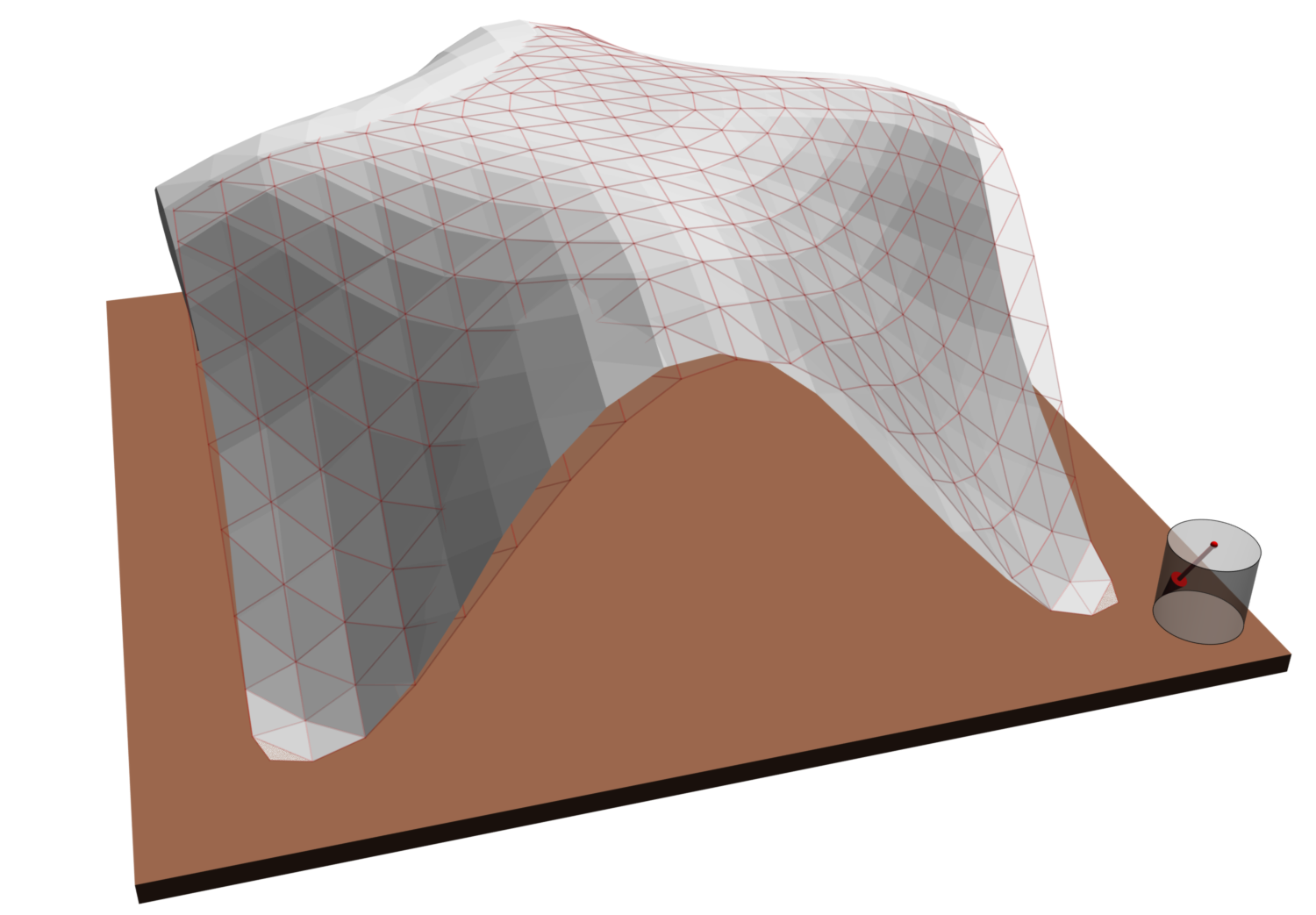}
	\includegraphics[width=0.24\textwidth]{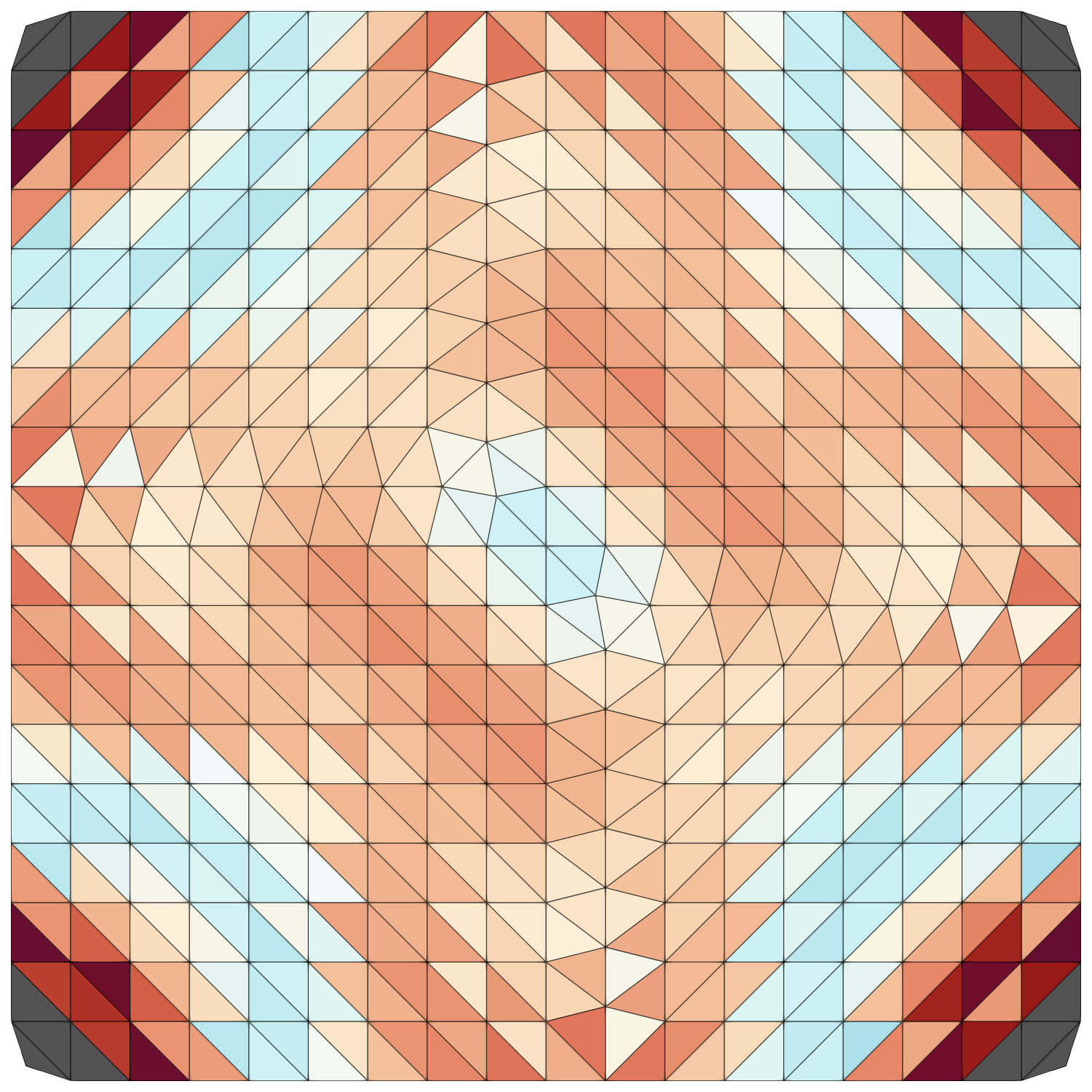}
	\includegraphics[width=0.24\textwidth]{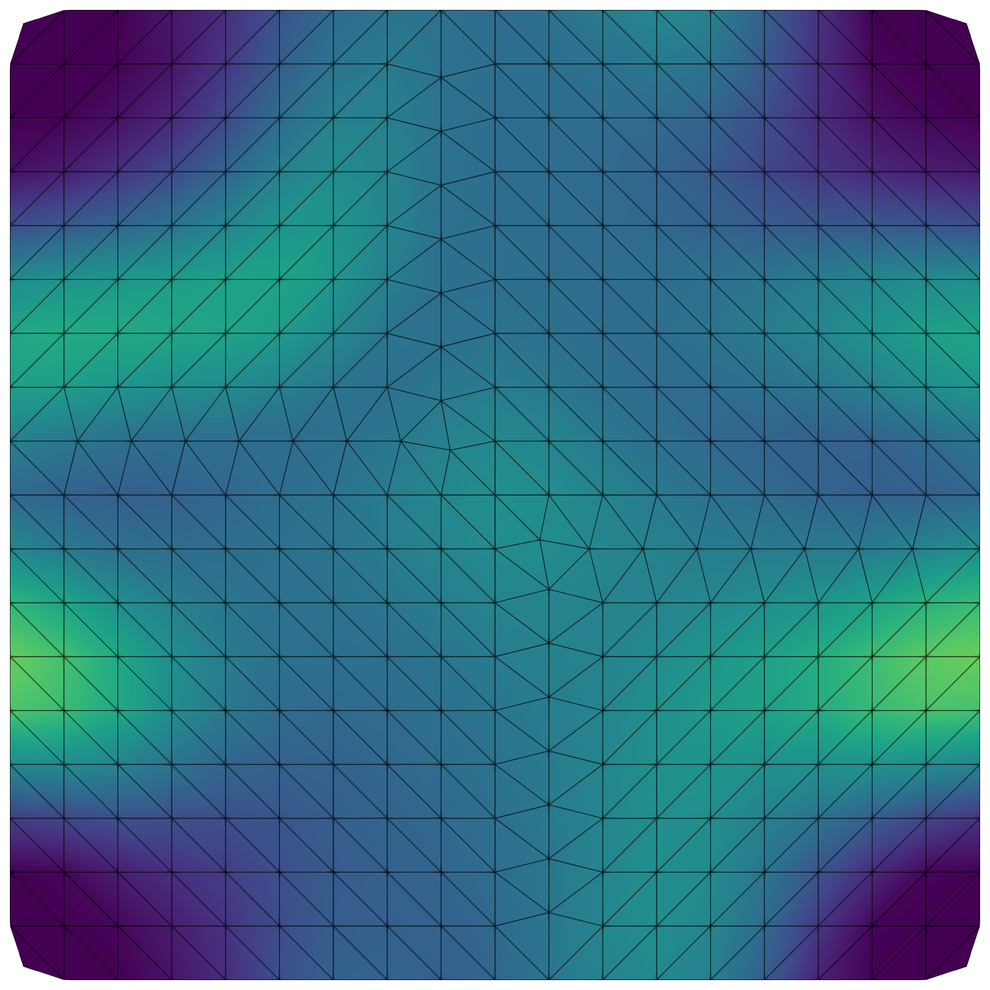}
	\includegraphics[width=0.05\textwidth]{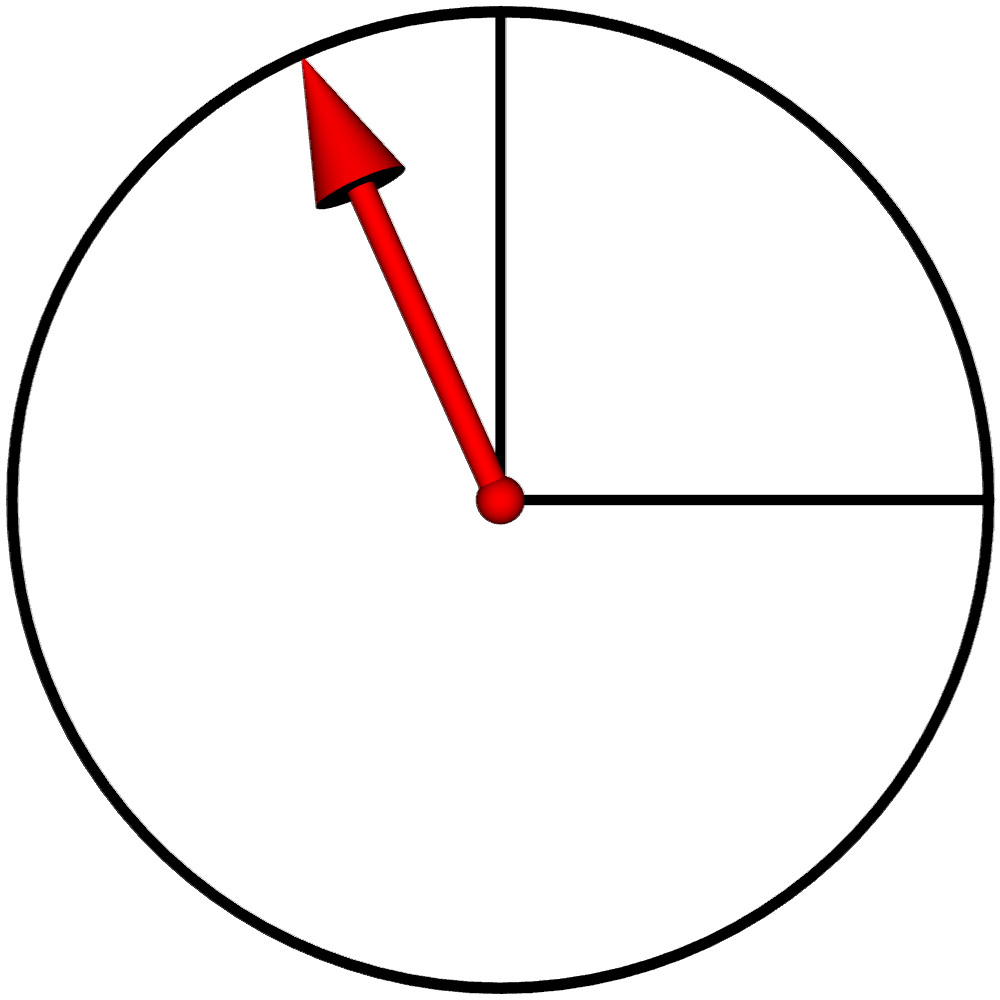}
	\includegraphics[width=0.45\textwidth]{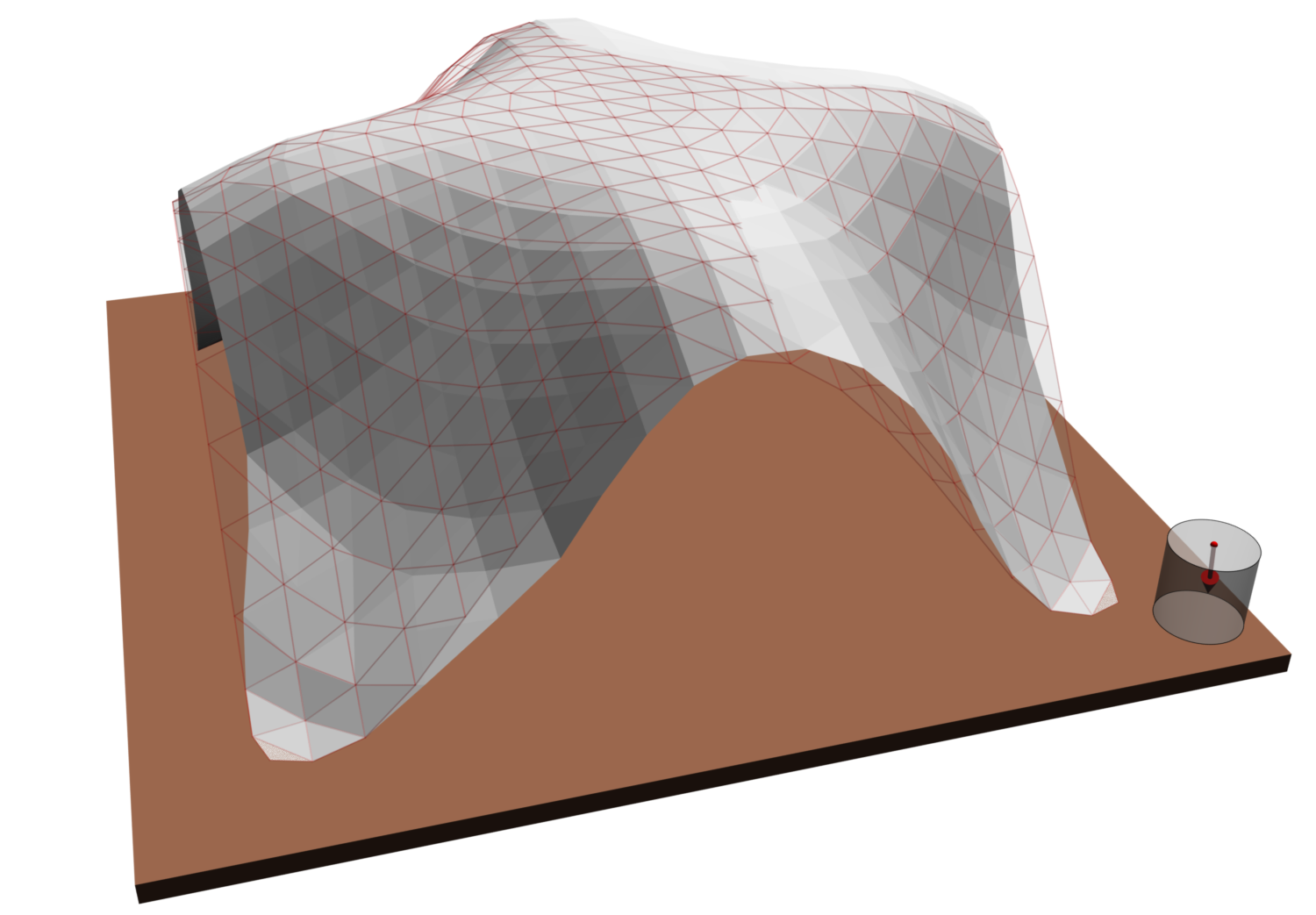}
	\includegraphics[width=0.24\textwidth]{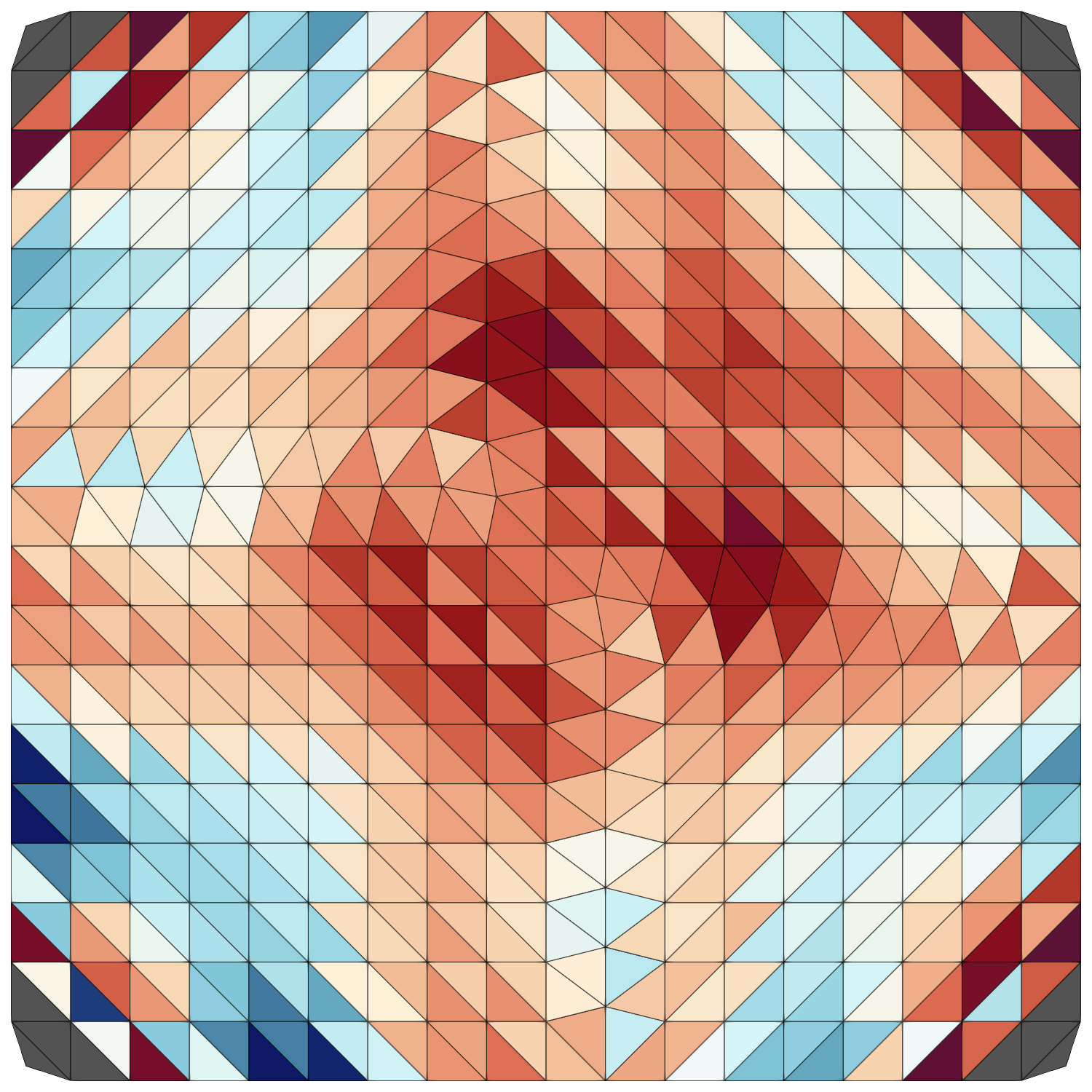}
	\includegraphics[width=0.24\textwidth]{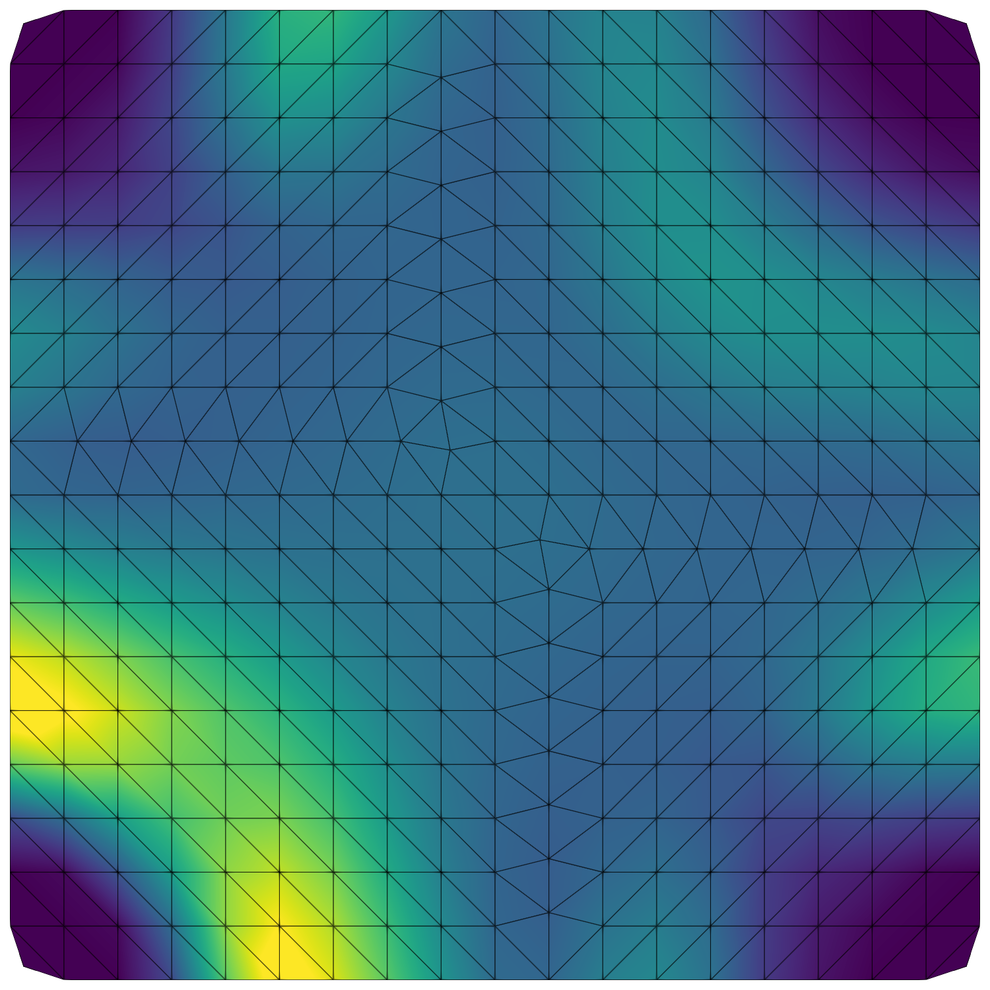}
	\includegraphics[width=0.05\textwidth]{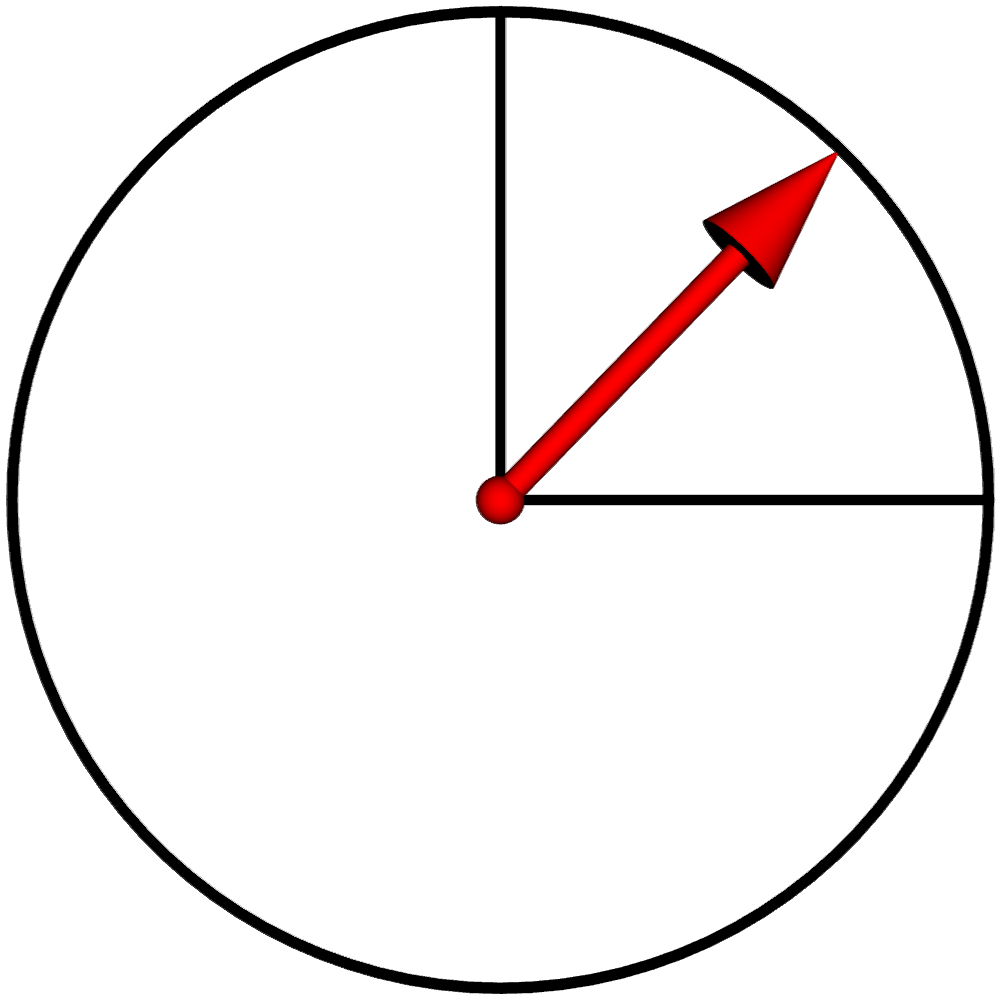}
	\caption{Comparison of results for full vertex tracking set (top) and plateau tracking set (bottom) on the simple roof-type geometry
        already shown in Figure~\ref{fig:setup}.
        On the left, we show the deformed configurations as gray surfaces, while the undeformed surfaces are shown as translucent surfaces overlayed with red edges. 
        Next to the surfaces, we visualize the direction of the force $(F_1,F_2,F_3)$ chosen by the follower in the cylinder of admissible values. 
        In the middle, we show the resulting material distributions with color map $0\hspace{1mm}$\protect\resizebox{.08\linewidth}{!}{\protect\includegraphics{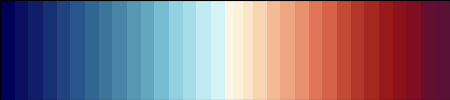}}$\hspace{1mm}0.2$, where boundary triangles with all three vertices subject to Dirichlet boundary conditions are shown in gray.
        On the right, we show the magnitude of the deformation \(\y\) using the color map  $0\hspace{1mm}$\protect\resizebox{.08\linewidth}{!}{\protect\includegraphics{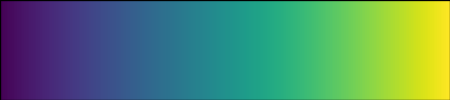}}$\hspace{1mm}\geq1.5$.
        Additionally, on the far right, we show the direction of the horizontal forces $(F_1,F_2)$.
	} \label{fig:toymodel}
\end{figure}
%fffffffffffffffffffffffffffffffffffffffff

In Figure~\ref{fig:toymodel}, we show the deformed configuration, the optimized distribution of the material thickness, and the magnitude of displacements in case of the leader minimizing a tracking functional once with global support (top row) using \(\chi\equiv 1\) and once restricted to the region of the roof plateau (bottom row). 
As for all examples presented here, in the follower problem, the maximal compliance is attained for a force $F$ representing an extremal point of the cylinder of admissible forces.
For the tracking cost domain centered on the roof plateau, one observes a concentration of mass in the central region accompanied by a significant reduction of the thickness close to the four corners where Dirichlet boundary conditions apply. 
The concentration and corresponding reduction break the symmetry of the configuration w.r.t.\ the diagonal from the upper left to the lower right.
Due to the asymmetric reduction, the follower chooses a force pointing to the upper right and one observes a kink line connecting the two arcs in the front at approximately half of the total height.
This is accompanied by large displacements, which are however outside of the tracking region on the plateau.
In contrast, for the tracking with global support, no such kink with strong displacements occurs, however, the deformation exhibits a larger  displacement in the central region.
Finally, beyond the mass concentration in the middle, one also observes the onset of curved ``beam'' like structures connecting the middle region and the four arcs of the roof.
In the example with localized tracking, and most of the following ones, the elementwise bounds \(\mat^+\) and \(\mat^-\) are nearly attained for at least some triangles. 

Figure~\ref{fig:massvariation} shows for the same geometry the impact of the upper bound on the total material volume. 
%fffffffffffffffffffffffffffffffffffffffff
\begin{figure}[t]
	\centering
	\includegraphics[width=0.17\textwidth]{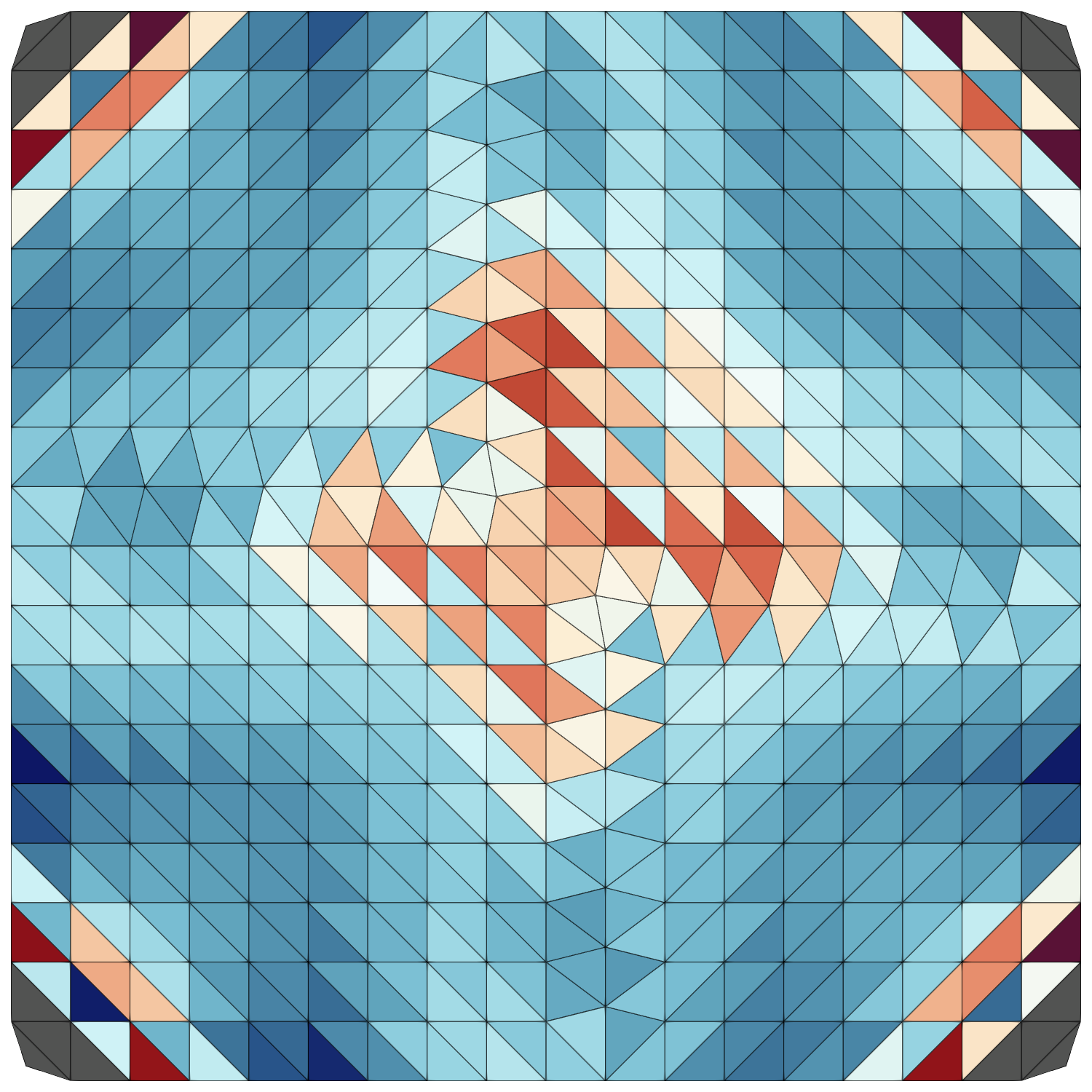}\hfill
	\includegraphics[width=0.17\textwidth]{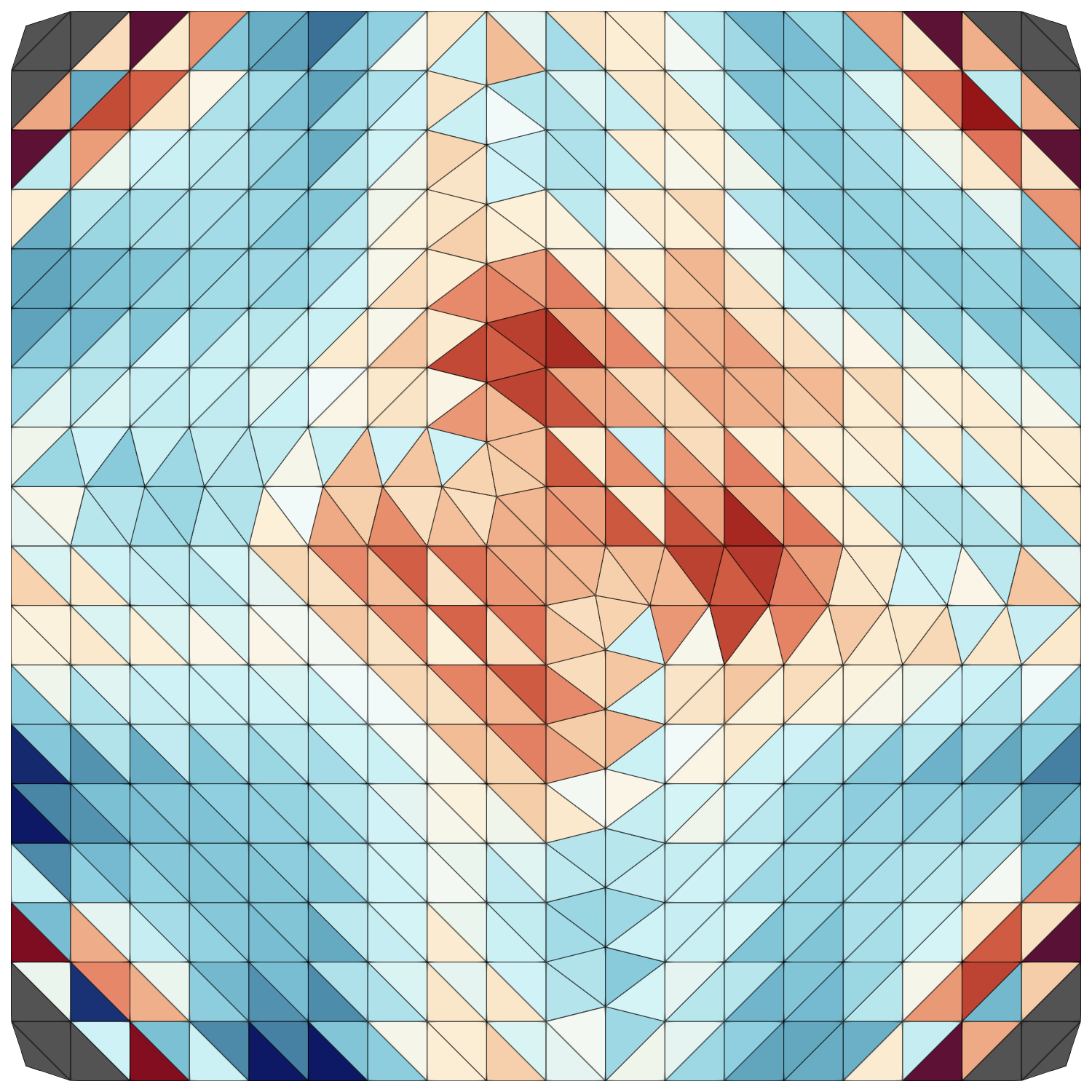}\hfill
	\includegraphics[width=0.17\textwidth]{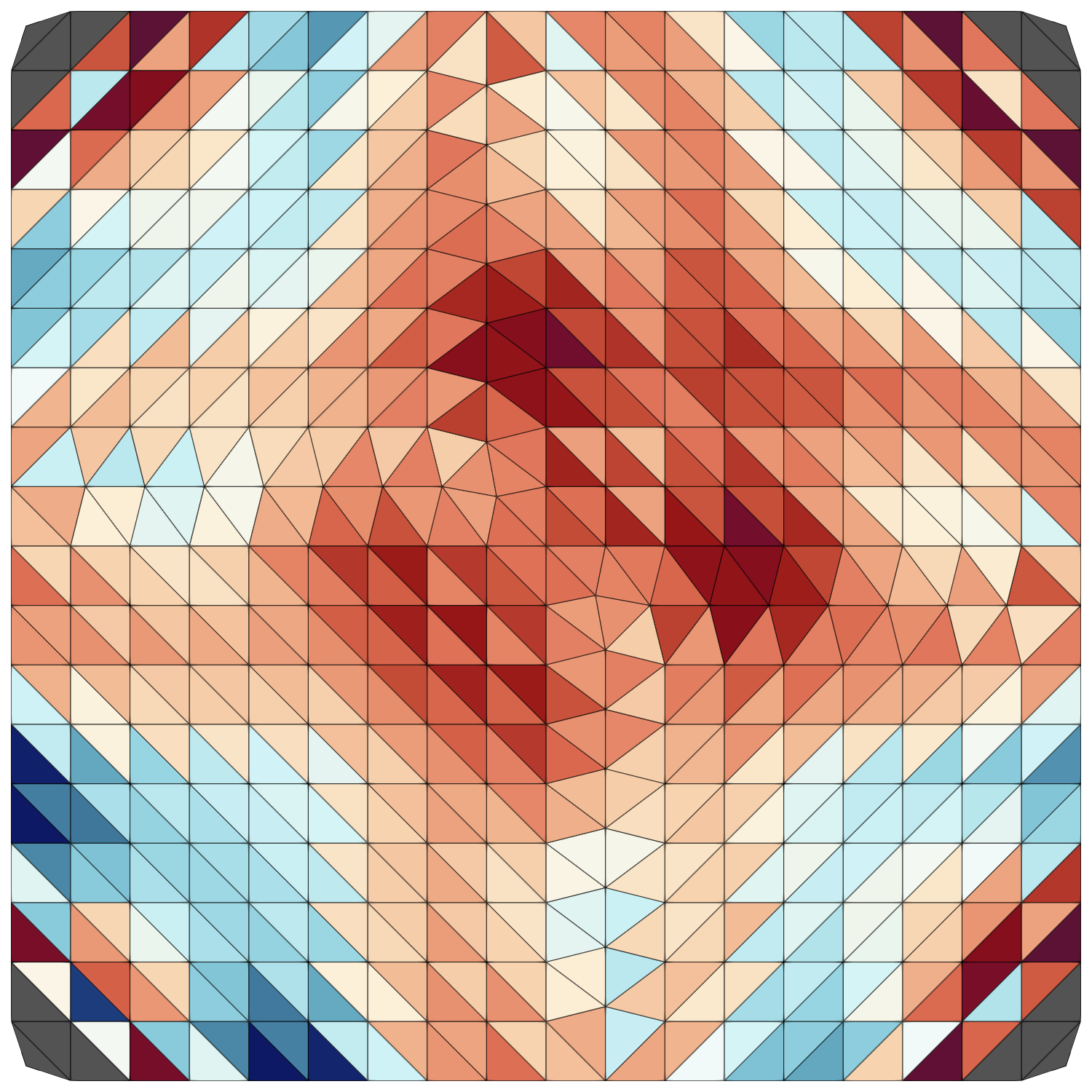}\hfill
	\includegraphics[width=0.17\textwidth]{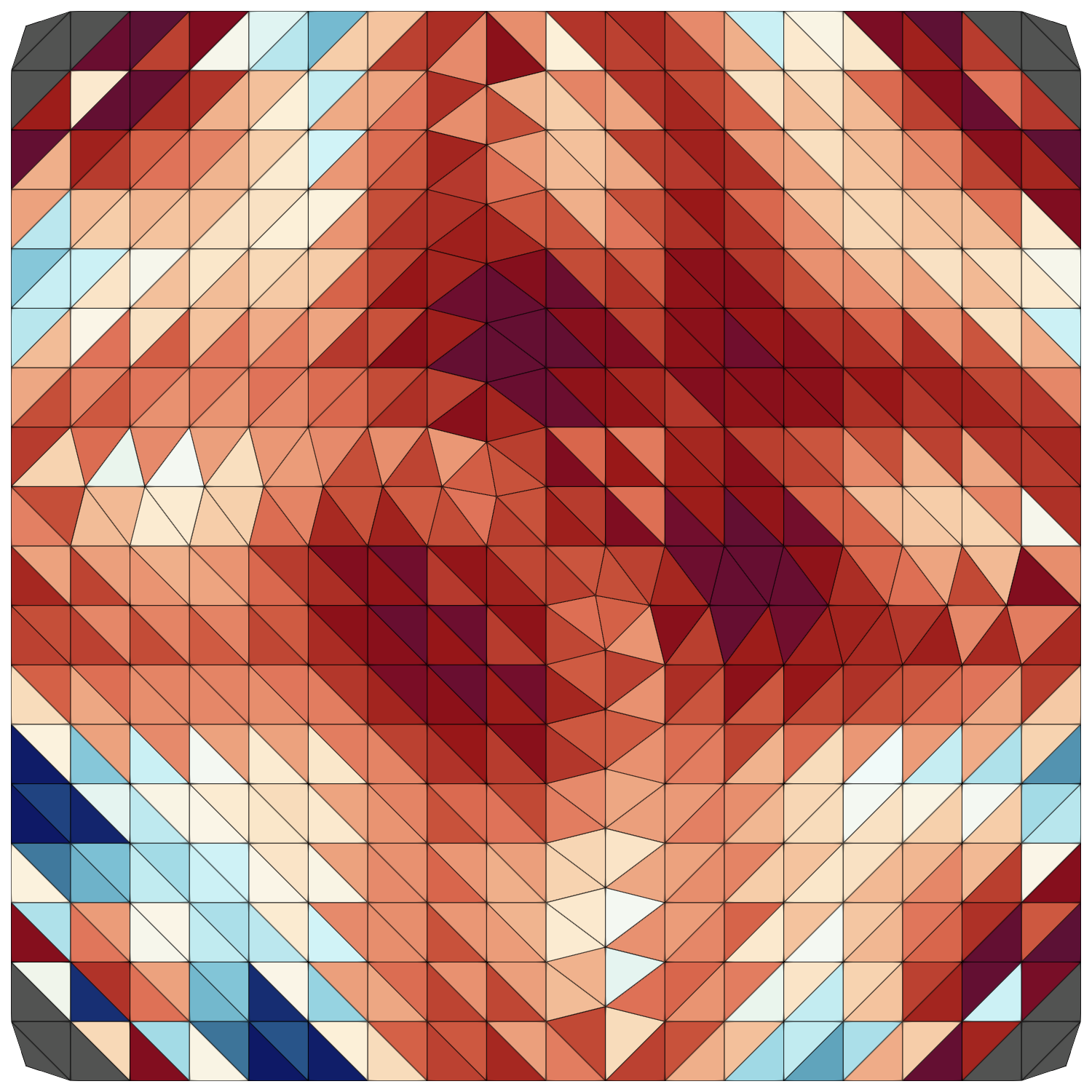}\hfill
	\includegraphics[width=0.17\textwidth]{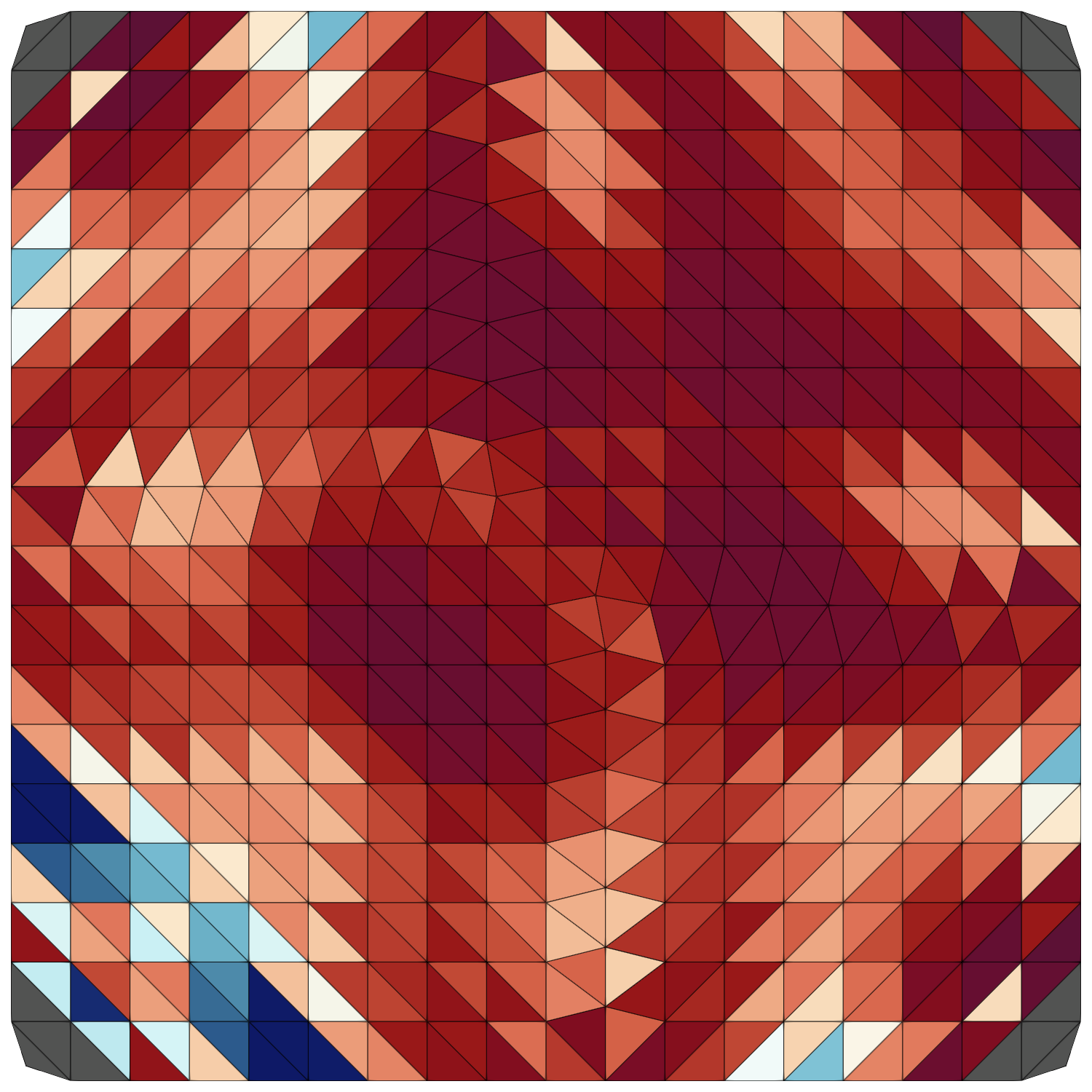}
	\includegraphics[width=0.05\textwidth]{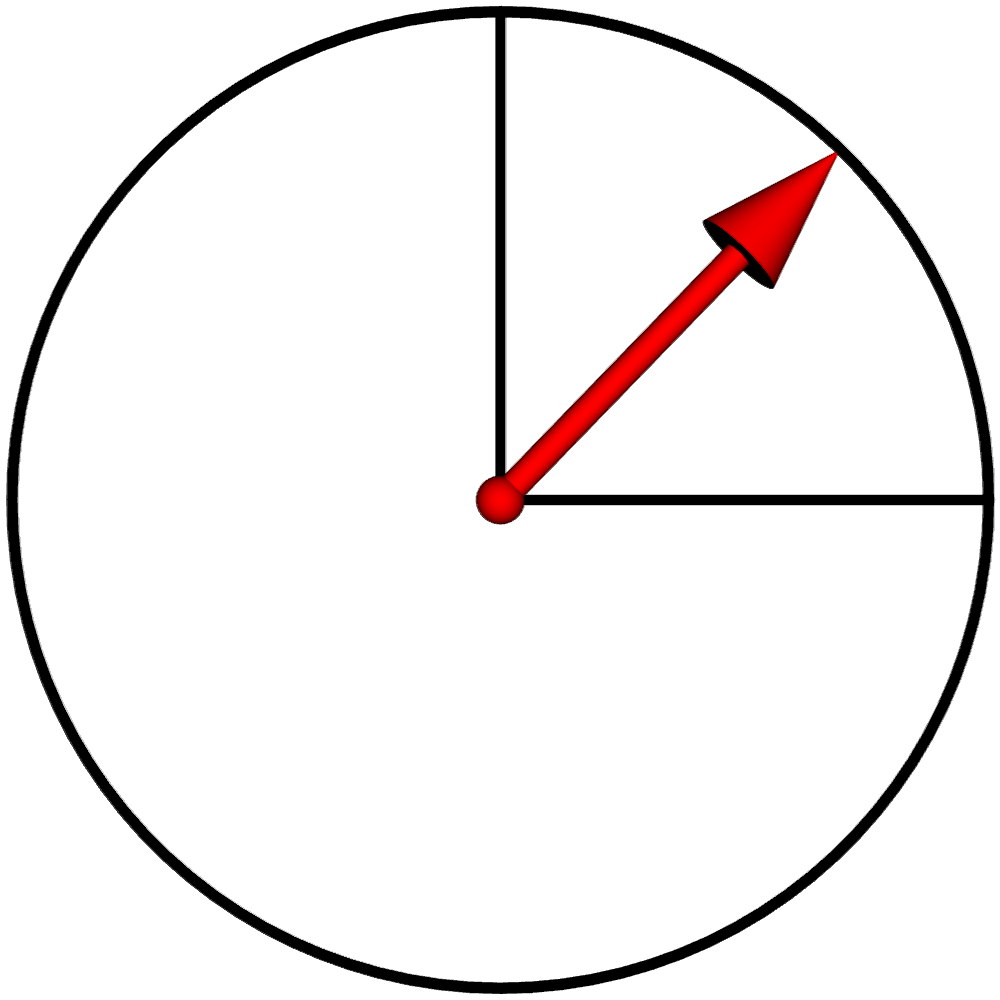}
	\caption{A comparison of the material distribution when varying the maximal allowed material volume \(V^+\) while keeping the other parameters fixed. The allowed volume was $V^+= 40,50,60,70,80$ from left to right.
		Material thickness is shown using the color map $0\hspace{1mm}$\protect\resizebox{.08\linewidth}{!}{\protect\includegraphics{figures/colorbar_cw.png}}$\hspace{1mm}0.2$. 
		On the far right, we show the direction of the horizontal forces, which was the same for all parameters, while the vertical force was always chosen maximal.
	} \label{fig:massvariation}
\end{figure}
%fffffffffffffffffffffffffffffffffffffffff
As the total permitted mass is increased, the elongated curved ``beams'' connecting the tracking region in the center with the four arcs become thicker. 
Once the maximal thickness is reached in the central region and along these ``beams'', further mass is invested to reinforce the regions close to the Dirichlet boundaries.
The curved carrier ``beams'' and the central region are again designed asymmetrically w.r.t.\ the diagonal from the upper left to the lower right leading the follower to push towards the upper right.

We next investigate the effect of the parameters characterizing the strength of the forces, ${F_{\mathrm{max},z}}$ and ${F_{\mathrm{max},xy}}$, while keeping the total amount of material constant.
By scaling invariance, it is natural to focus on the ratio $\frac{F_{\mathrm{max},z}}{F_{\mathrm{max},xy}}$.
In Figure~\ref{fig:cylindervariation}, we show that with increasing strength of the vertical force, the ``beams'' become thinner and instead more material is concentrated in the central region.
%fffffffffffffffffffffffffffffffffffffffff
\begin{figure}[t]
	\centering
	\includegraphics[width=0.15\textwidth]{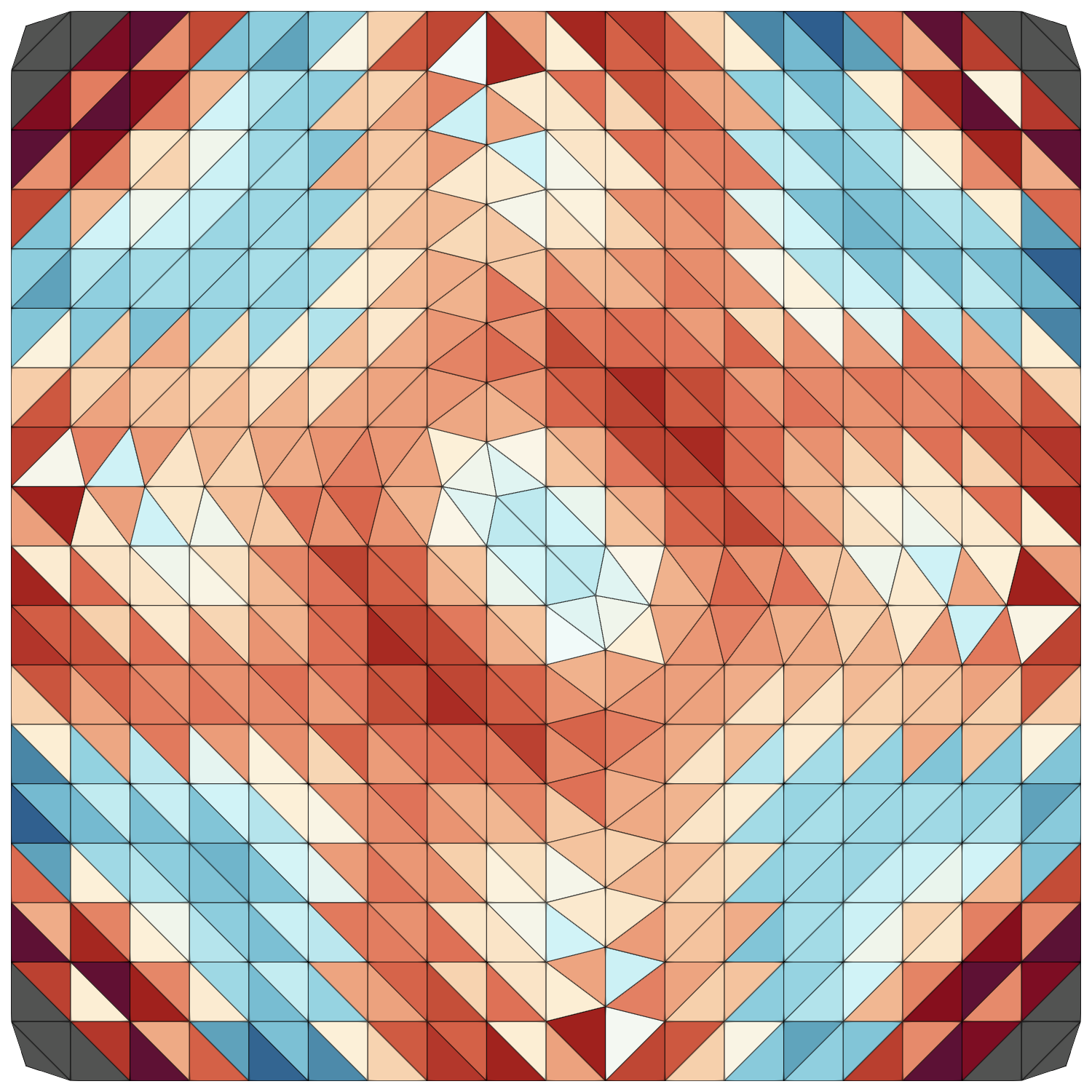}
	\includegraphics[width=0.03\textwidth,trim=100 0 100 0, clip]{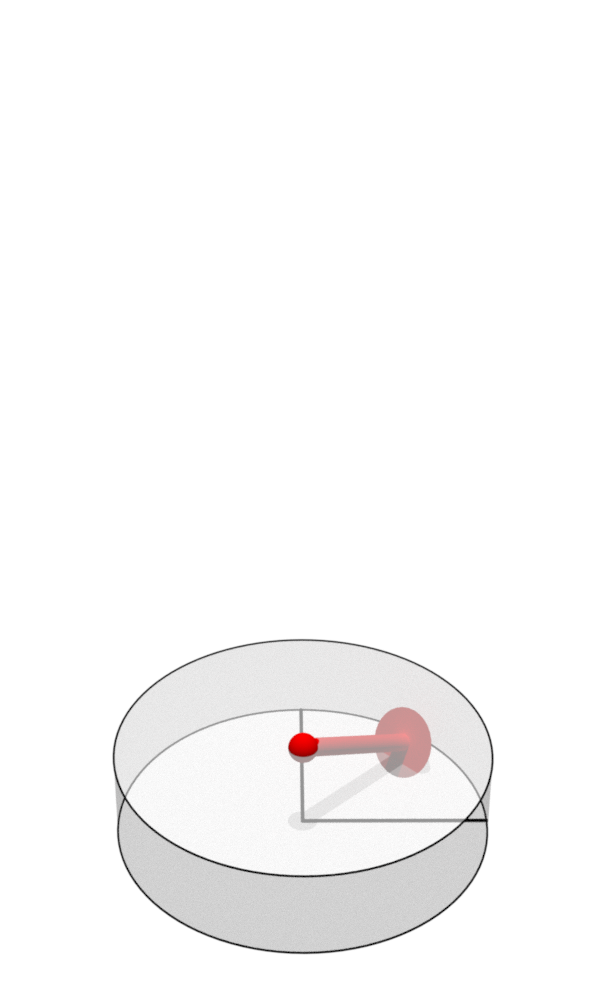}
	\includegraphics[width=0.15\textwidth]{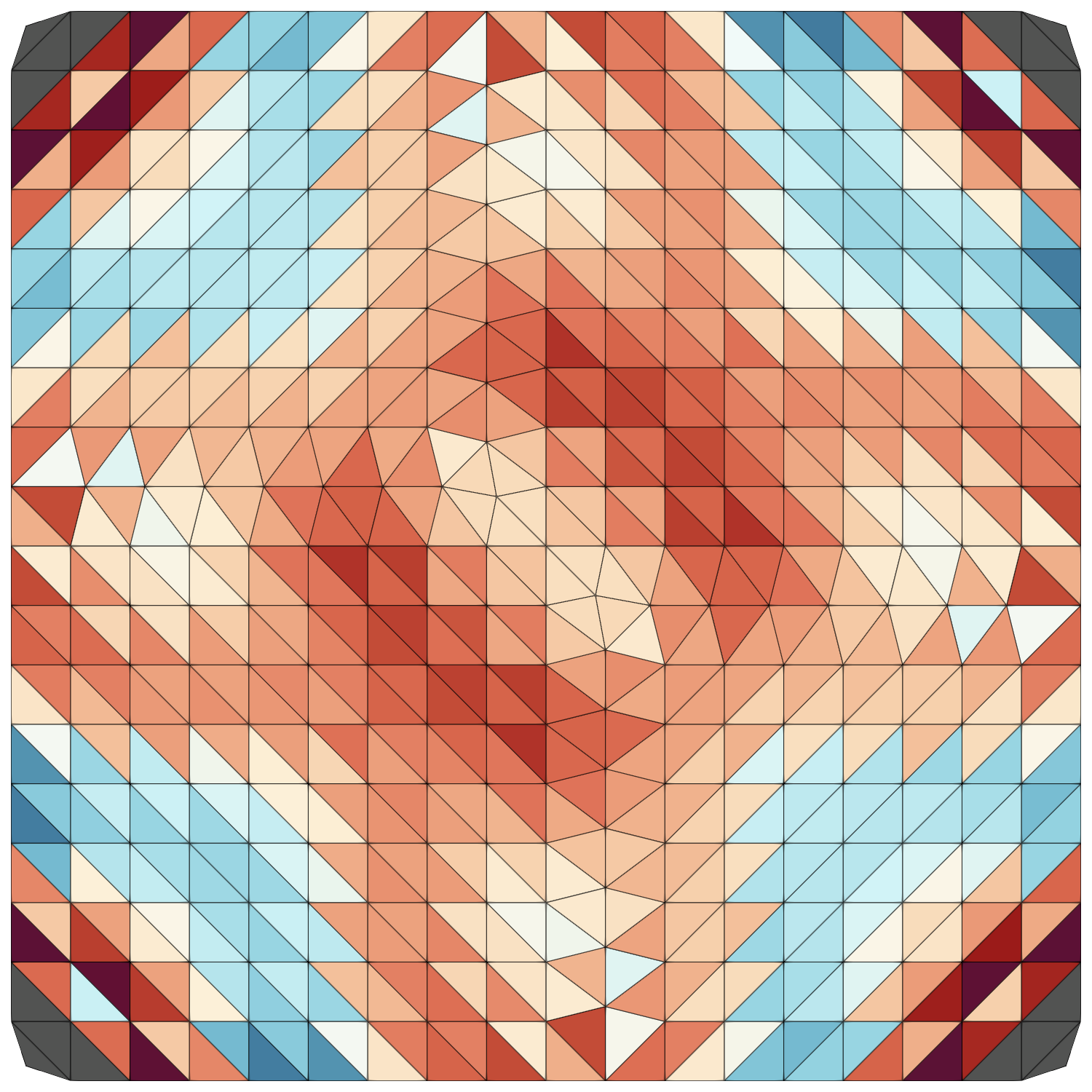}
	\includegraphics[width=0.03\textwidth,trim=100 0 95 0, clip]{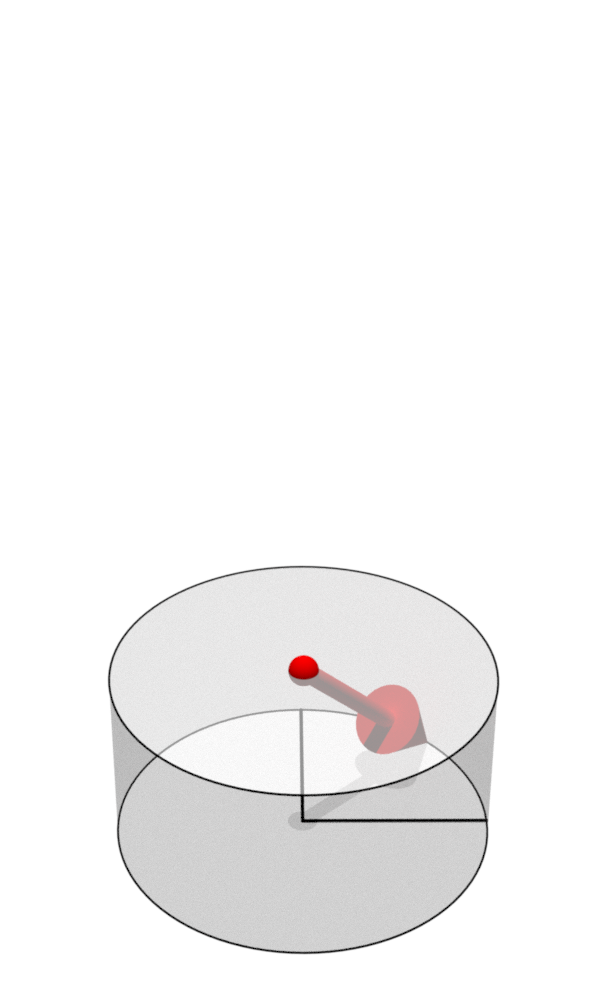}
	\includegraphics[width=0.15\textwidth]{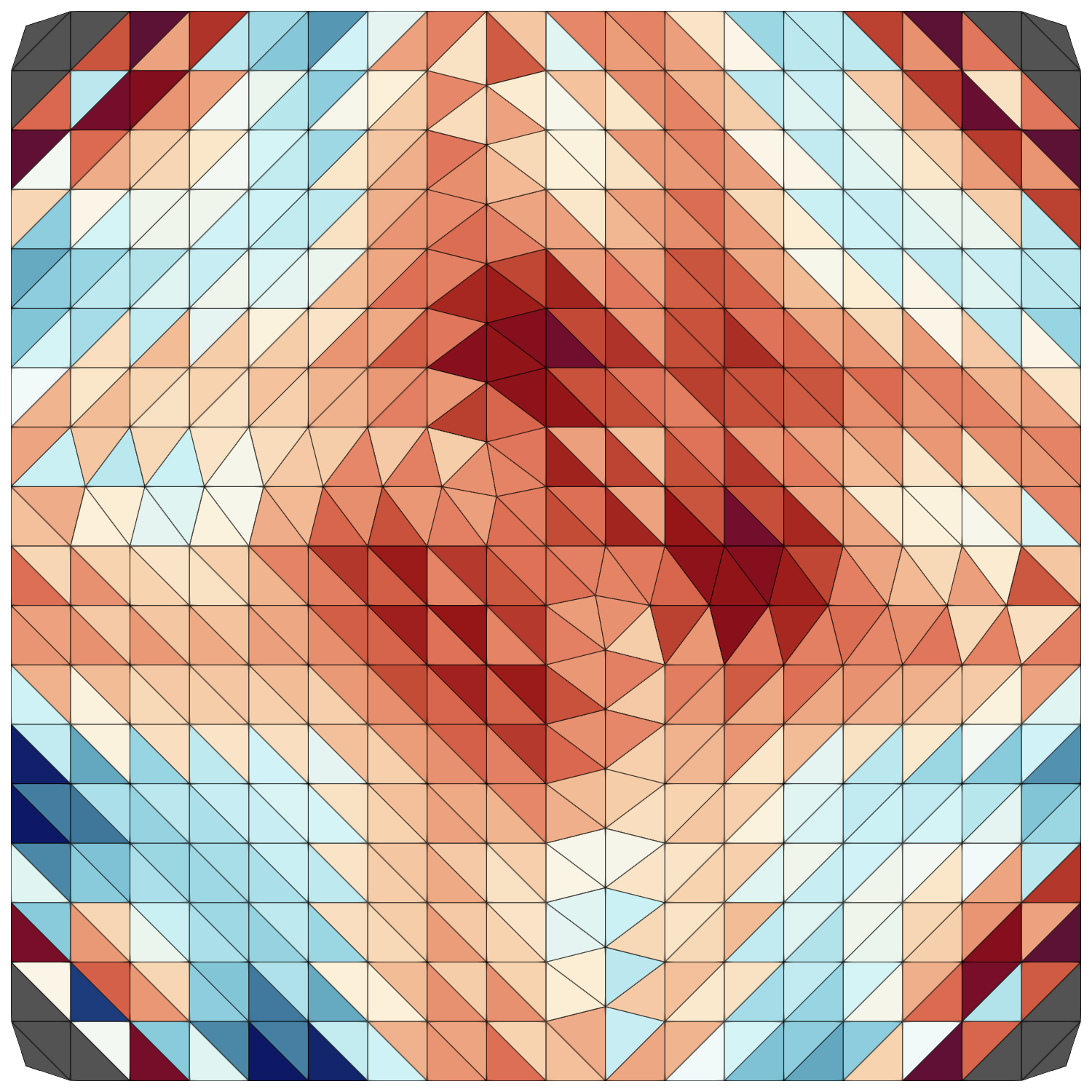}
	\includegraphics[width=0.03\textwidth,trim=85 0 80 0, clip]{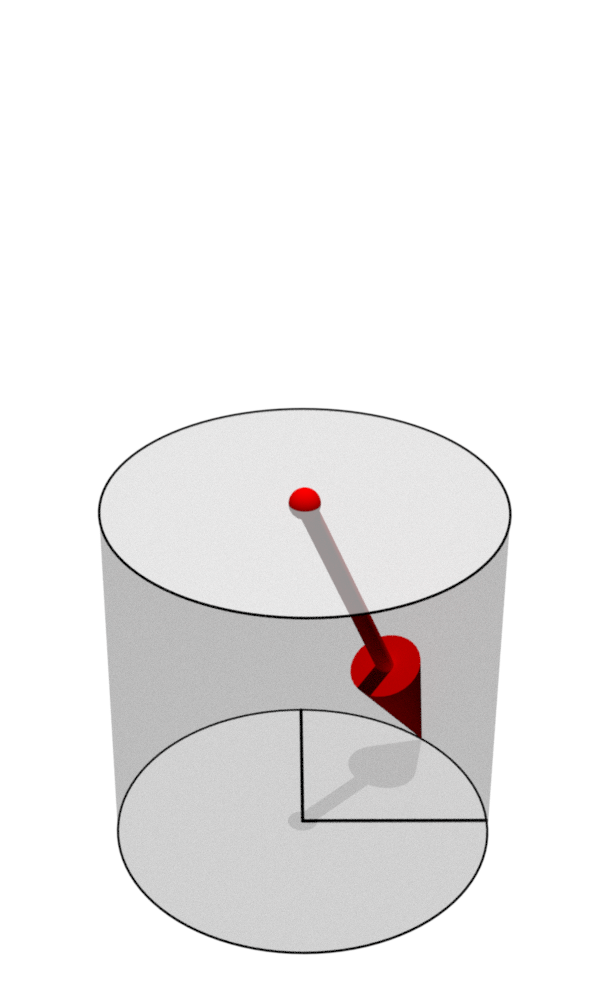}
	\includegraphics[width=0.15\textwidth]{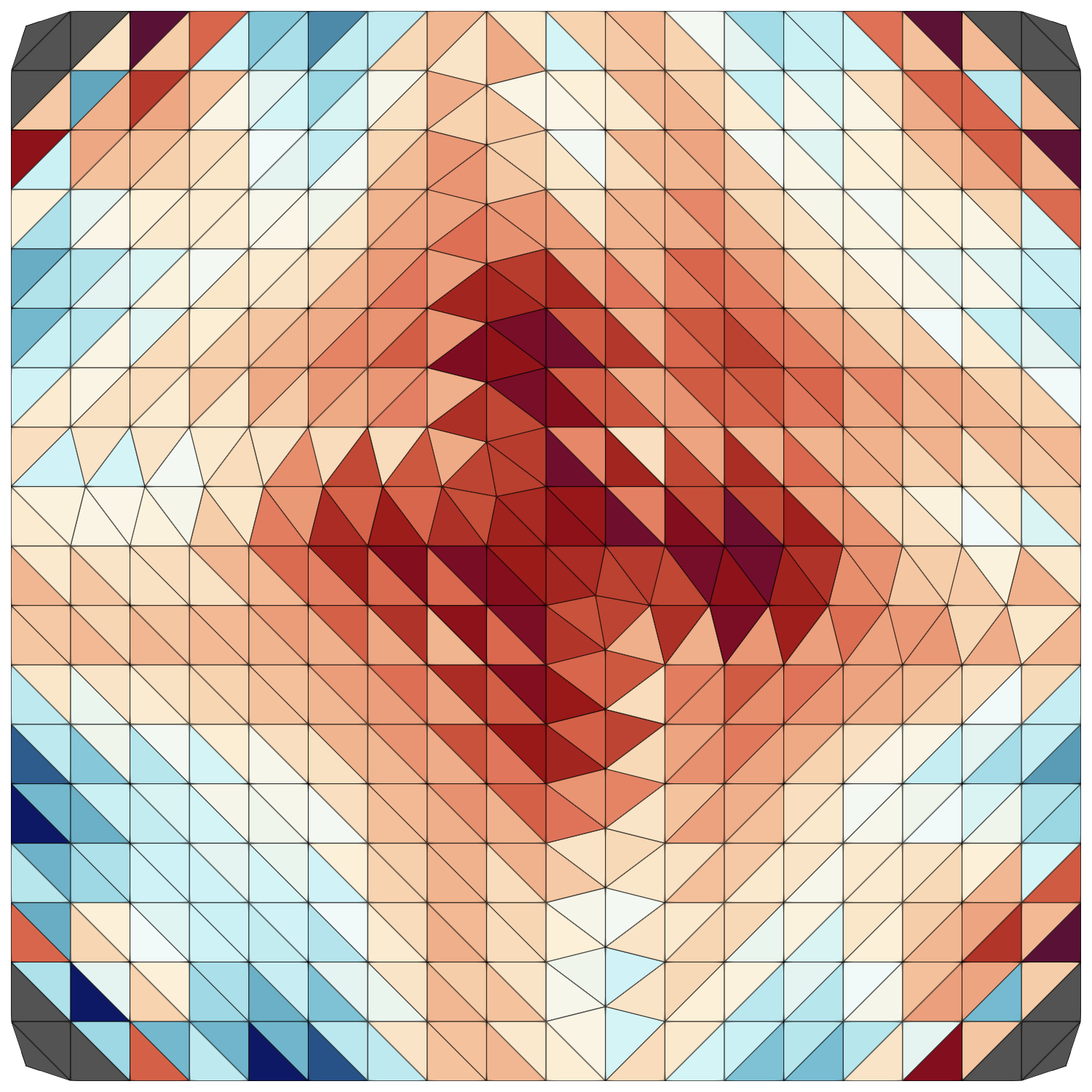}
	\includegraphics[width=0.03\textwidth,trim=70 0 60 0, clip]{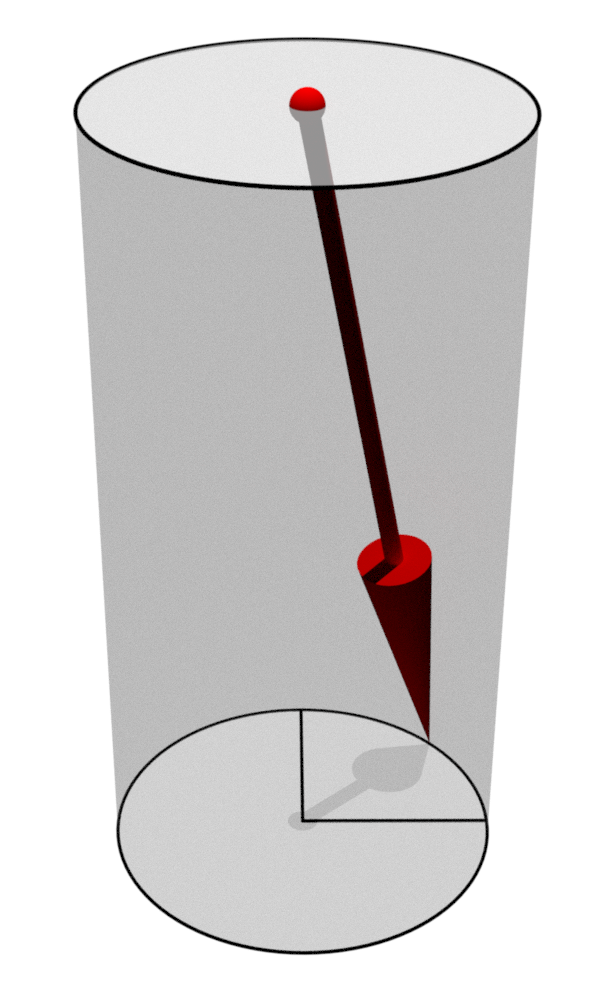}
	\includegraphics[width=0.15\textwidth]{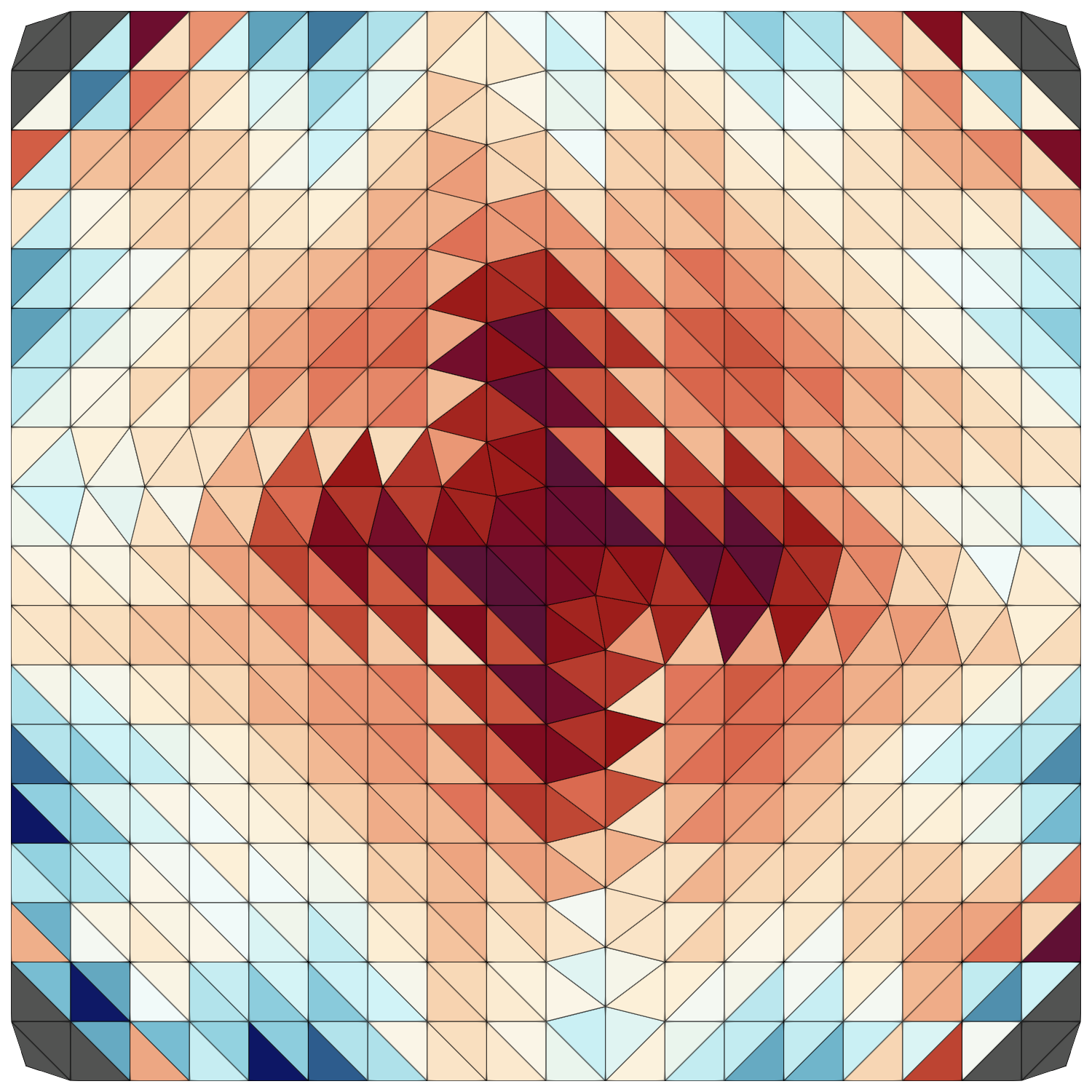}
	\includegraphics[width=0.05\textwidth,trim=160 100 160 0, clip]{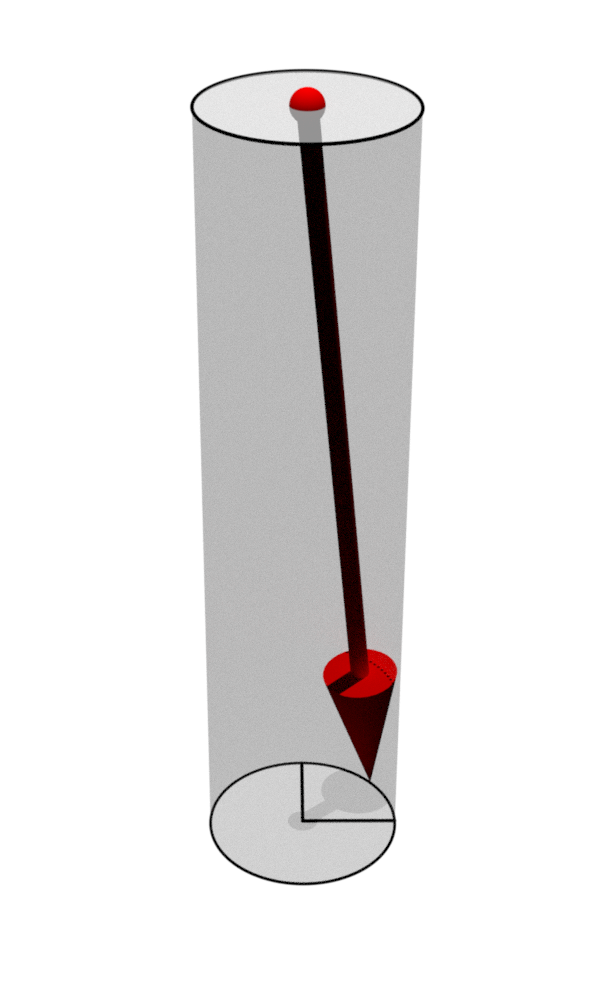}
	\caption{
		A comparison of the material distribution when varying the ratio of vertical to horizontal force $\frac{F_{\mathrm{max},z}}{F_{\mathrm{max},xy}}$, \ie the shape of the cylinder, while keeping the other parameters, especially the maximal magnitude of horizontal force, fixed. 
		The ratio of vertical to horizontal force was $\frac{F_{\mathrm{max},z}}{F_{\mathrm{max},xy}} = \frac{1}{2},1,2,4,8$ from left to right.
		The material thickness is shown using the color map $0\hspace{1mm}$\protect\resizebox{.08\linewidth}{!}{\protect\includegraphics{figures/colorbar_cw.png}}$\hspace{1mm}0.2$. 
On the right of each material distribution, we show the force
		in the cylinder of admissible values.
		} \label{fig:cylindervariation}
\end{figure}
%fffffffffffffffffffffffffffffffffffffffff
Interestingly, for small values of the ratio between the two forces the material distribution is nearly symmetrical w.r.t.\ the diagonal from the upper left to the lower right, while it is asymmetric for mid-range ratios and then becomes more symmetric again for large ratios.

Figure~\ref{fig:deviation} shows the impact of the strength of the  stochastic perturbation of the material thickness, as measured by the standard deviation, again for the tracking region on the roof plateau. 
%fffffffffffffffffffffffffffffffffffffffff
\begin{figure}[t]
	\centering
	\includegraphics[width=0.17\textwidth]{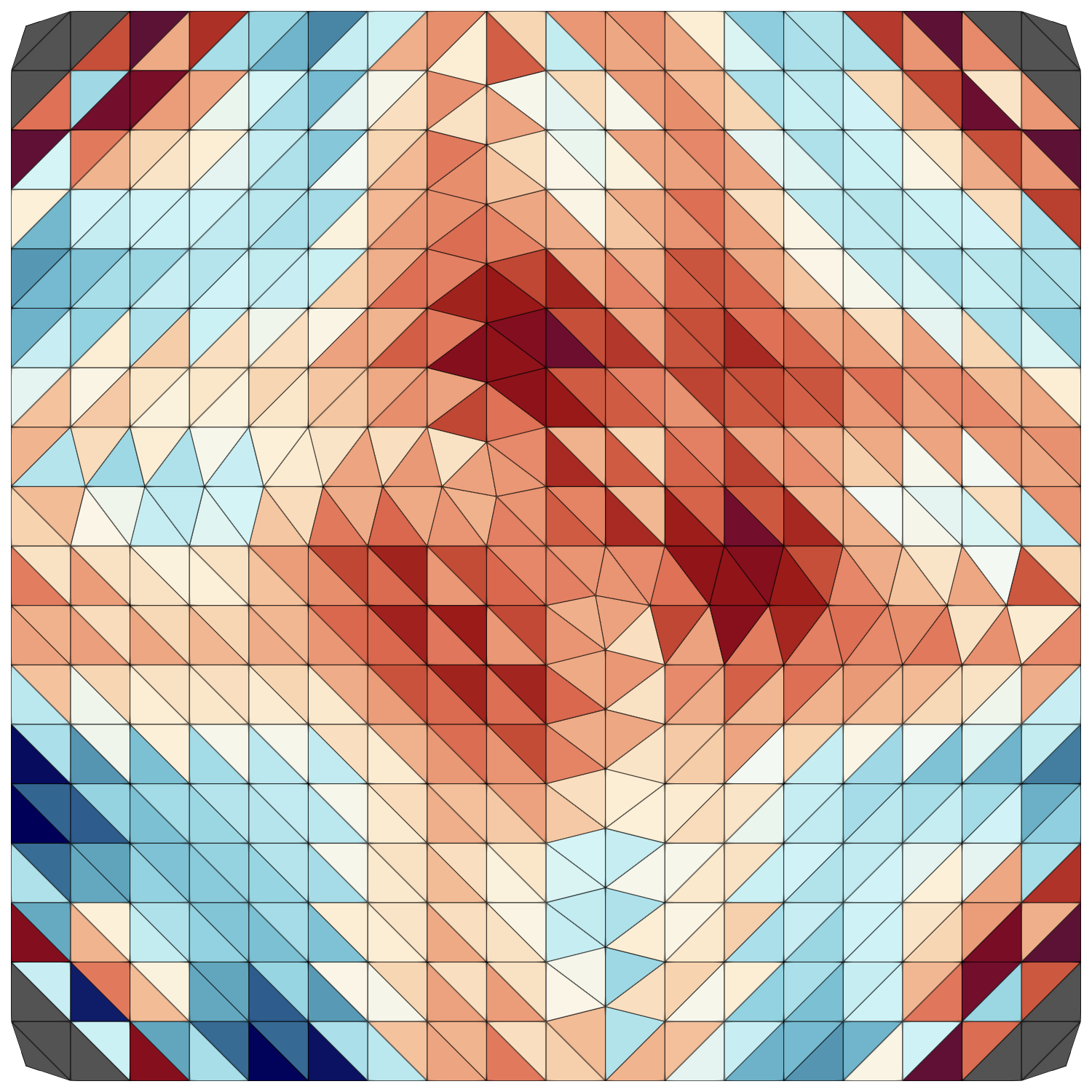}
	\includegraphics[width=0.17\textwidth]{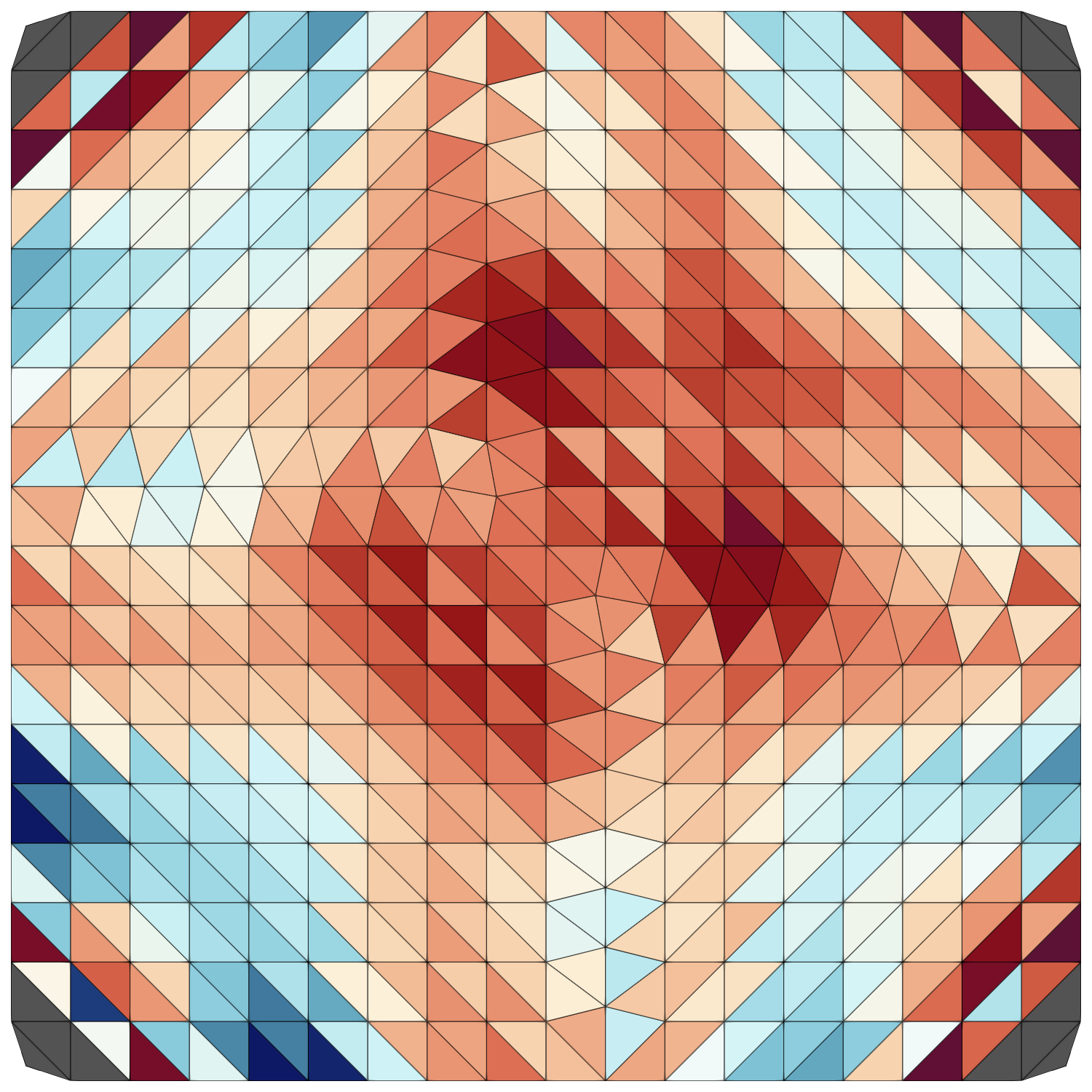}
	\includegraphics[width=0.17\textwidth]{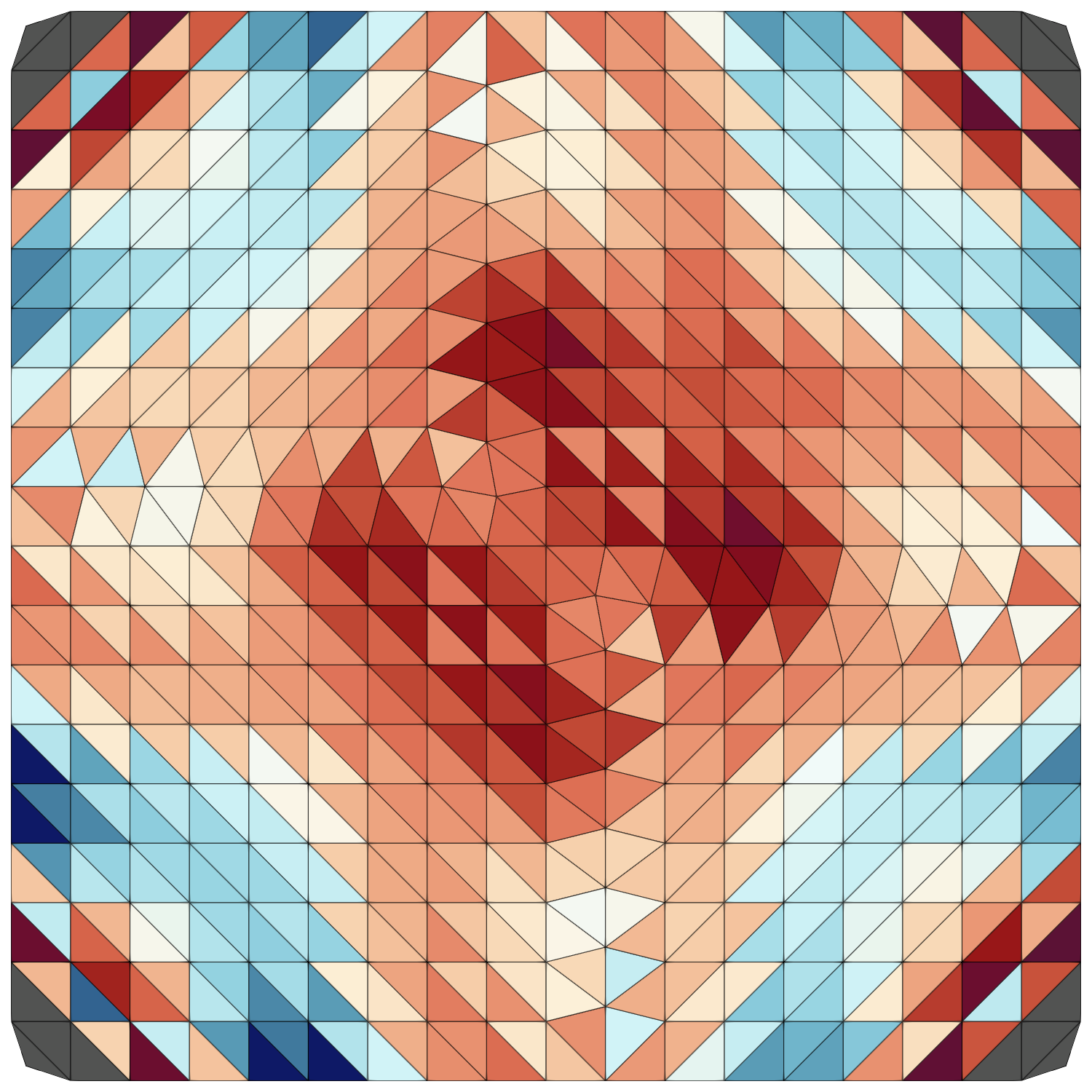}
	\caption{Comparison of material distribution when varying the standard deviation $\sigma$ of the material perturbation while keeping the other parameters fixed.
		The standard deviation was 
		$\sigma = \frac{5}{100}, \frac{1}{10}, \frac{2}{10}$ from left to right.
		Material thickness is shown using the color map $0\hspace{1mm}$\protect\resizebox{.08\linewidth}{!}{\protect\includegraphics{figures/colorbar_cw.png}}$\hspace{1mm}0.2$. 
	} \label{fig:deviation}
\end{figure}
%fffffffffffffffffffffffffffffffffffffffff
With increasing strength of the stochastic perturbation, the optimal structure becomes more diffuse.
Indeed, in a deterministic setting, the leader could aim for a finely-structured design, but very imprecise manufacturing is likely to render it ineffective. 
In order to understand this effect, we describe an idealized situation: If the leader concentrates mass on a single row of $k$ elements, then a large negative fluctuation in the thickness of any single one of them is sufficient to destroy the strength of the construction. 
If instead, the leader distributes the mass on $k^2$ elements filling a square, then at least a number of order $k$ of those (with a specific geometry, for example, a column) must have a large negative fluctuation before the structure loses significantly in strength.

Lastly, in Figure~\ref{fig:complex}, we show two more complex examples of architectural designs of roof structures, inspired by \cite{VoHoWa12}. 
In the top row, we use a closed hall as the reference geometry for our bilevel optimization problem, which fills a box of approximately \(20  \times 20 \times 5\).
We limit the horizontal load with \(F_{\mathrm{max},xy} = 0.005\) and the vertical load with \(F_{\mathrm{max},z} = 2 F_{\mathrm{max},xy}\).
The elementwise bounds on the material thickness are \(\mat^- = 0.01\) and \(\mat^+ = 0.2\). 
The volume of the material is bounded by \(V^+ = 50\) and the stochastic variation is \(\sigma = 0.05\).  
The weights of the barrier terms are \(\alpha^F = 10^{-4}\), \(\alpha^u = 1\), and \(\alpha^V = 10^{-3}\).
In the bottom row, we use a reference geometry resembling a double torus cut in half, which fills a box of approximately \(70  \times 50 \times 15\).
Again, we limit the horizontal load with \(F_{\mathrm{max},xy} = 0.005\) and the vertical load with \(F_{\mathrm{max},z} = 2 F_{\mathrm{max},xy}\).
The elementwise bounds on the material thickness are again
\(\mat^- = 0.01\) and \(\mat^+ = 0.2\). 
The volume of the material is bounded by \(V^+ = 330\) and the stochastic variation is  \(\sigma = 0.05\).  
The weights of the barrier terms are \(\alpha^F = 10^{-4}\), \(\alpha^u = 1\), and \(\alpha^V = 10^{-1}\).
In both cases, we use the full domain as tracking set.
The main weakness of both structures is the concavity in the central part, which can be easily deformed by the vertical force. 
Hence, in both optimized solutions, the material is redistributed to prevent this. 
In the first case, this is done by building a stabilized ledge around the center, while in the second case beam-like structures from the two ``holes'' and another beam from the curve in the bottom emerge.
Furthermore, in the second one, also the ``entrance'' is stabilized by adding material at the ends of its arch.

%fffffffffffffffffffffffffffffffffffffffff
\begin{figure}[t]
	\centering
	\includegraphics[width=0.45\textwidth]{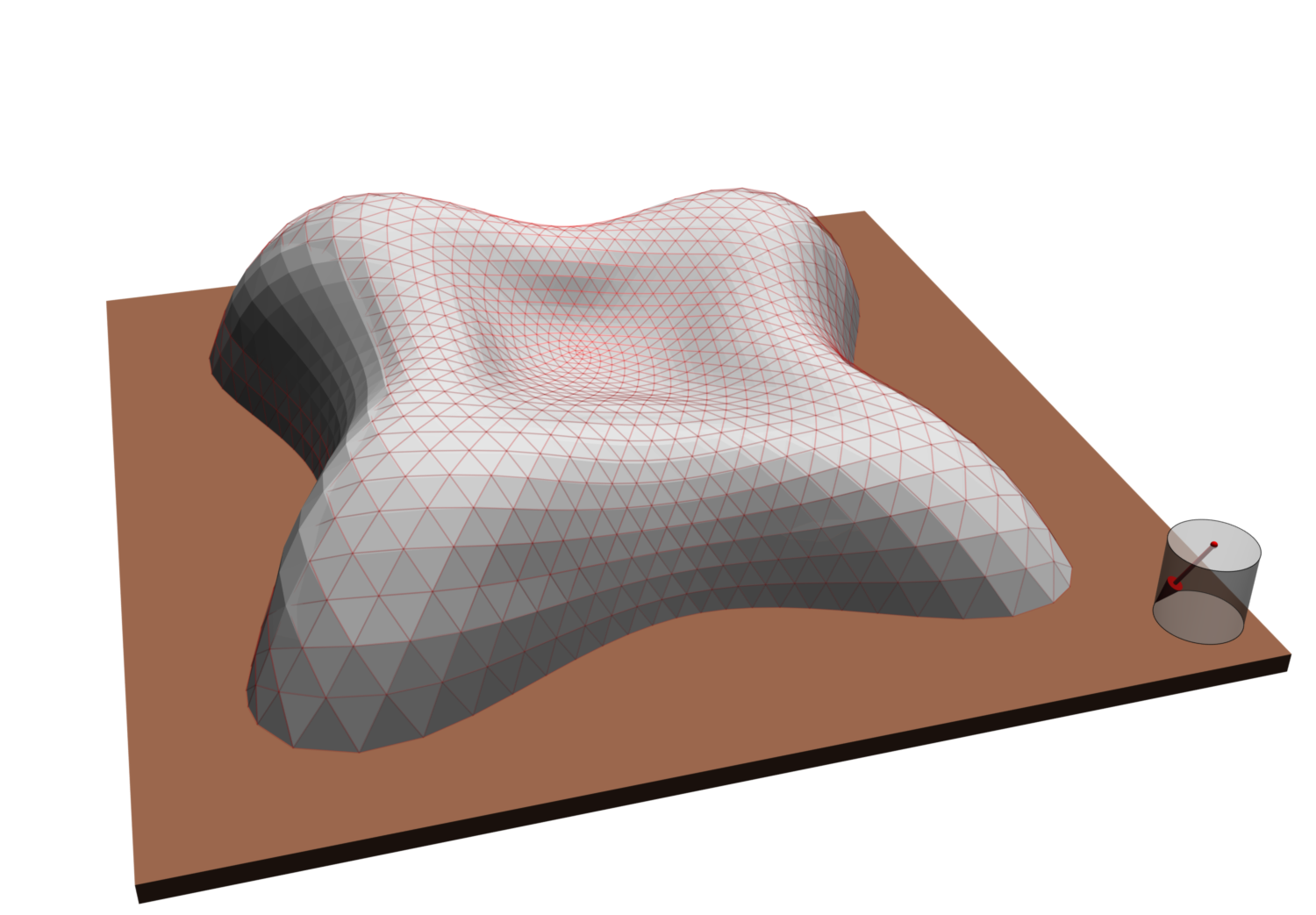}
	\includegraphics[width=0.24\textwidth]{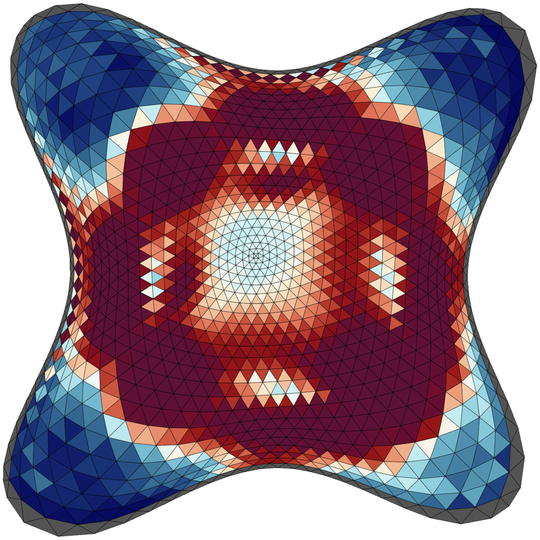}
	\includegraphics[width=0.24\textwidth]{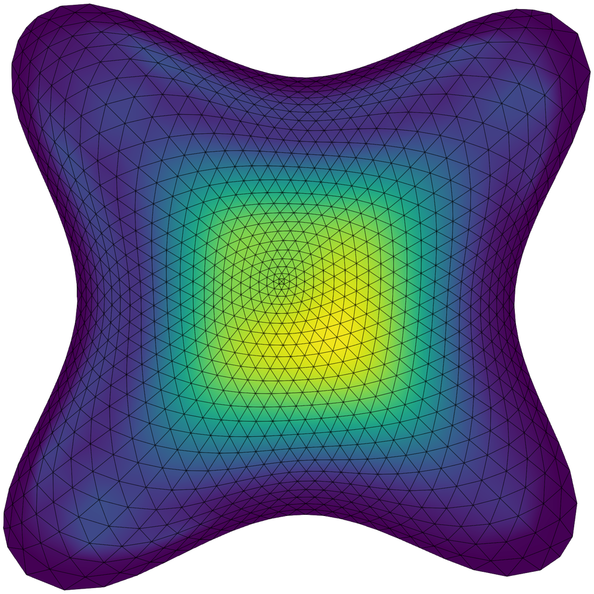}
	\includegraphics[width=0.05\textwidth]{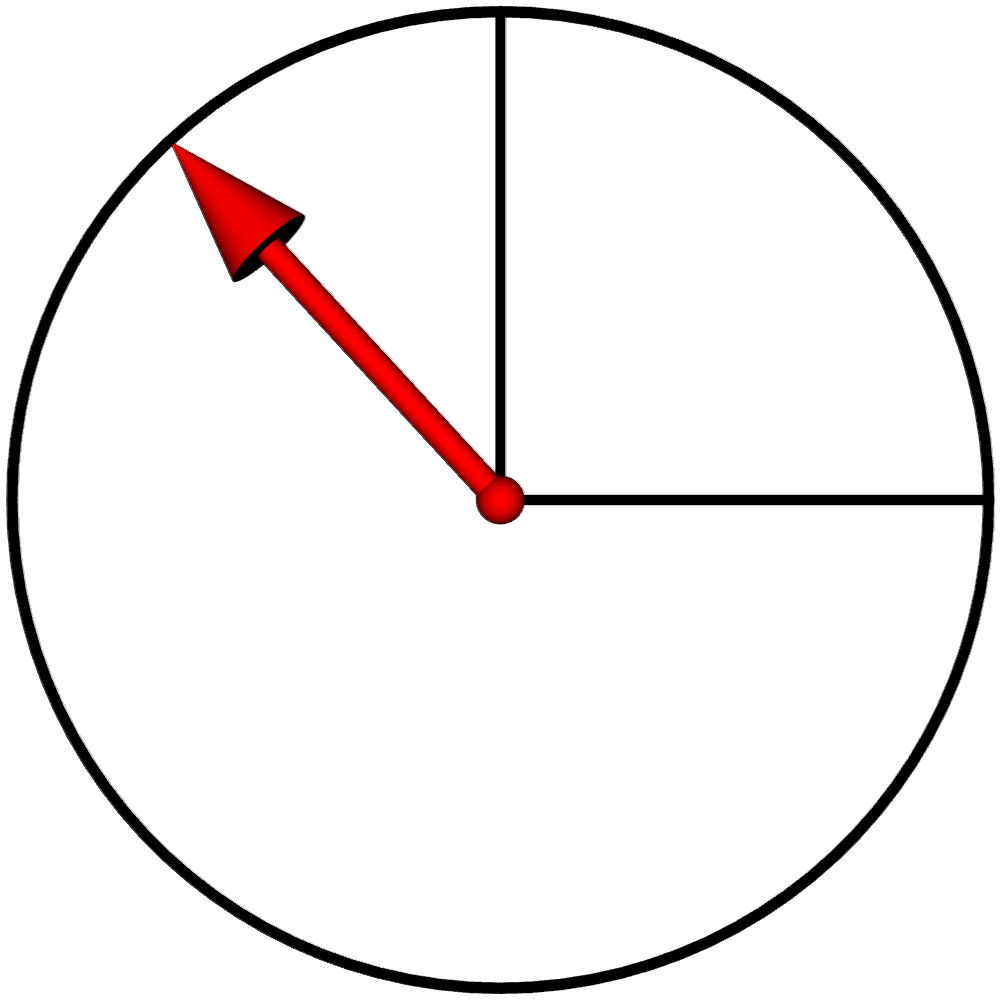}
	\includegraphics[width=0.45\textwidth]{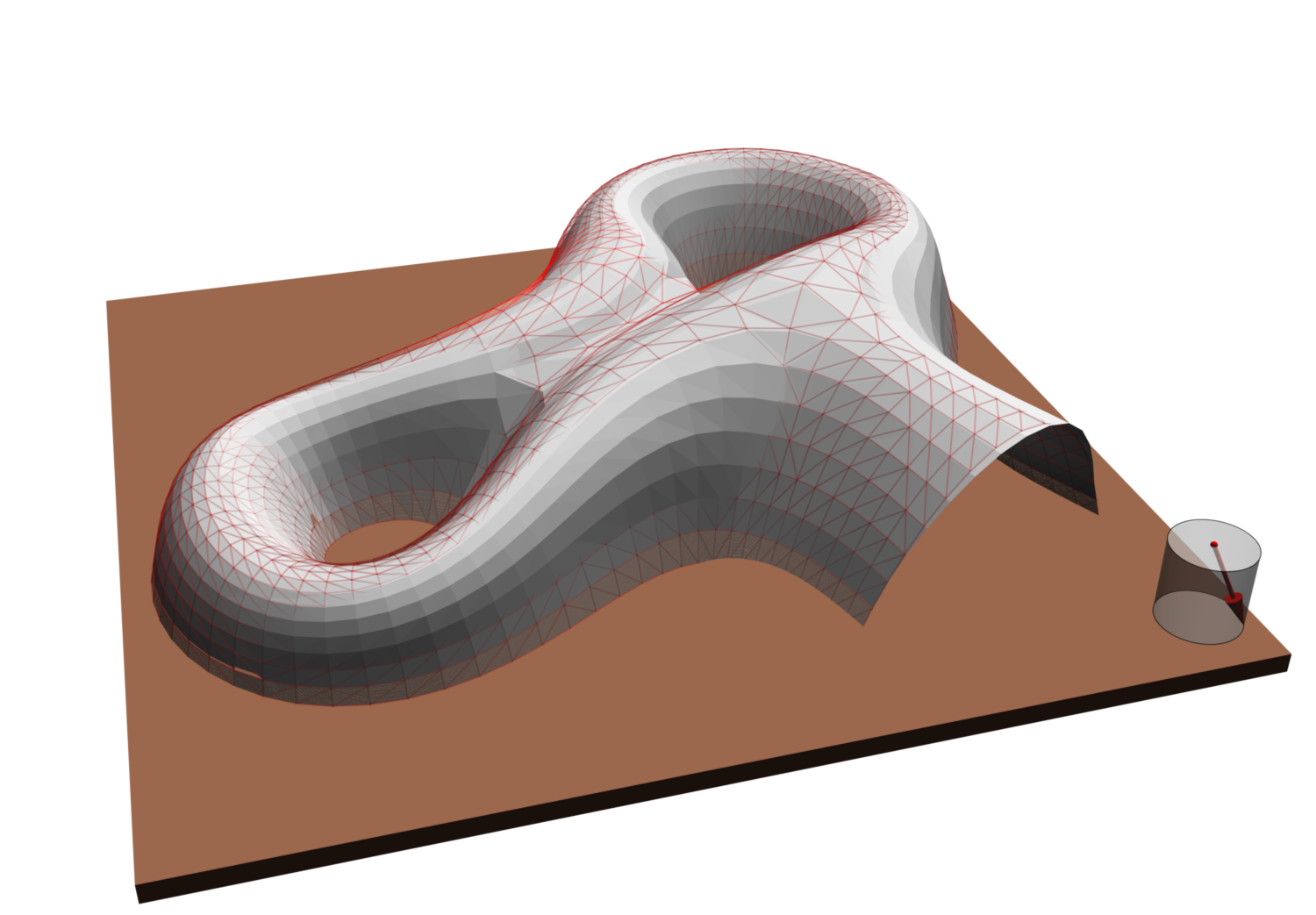}
	\includegraphics[width=0.24\textwidth]{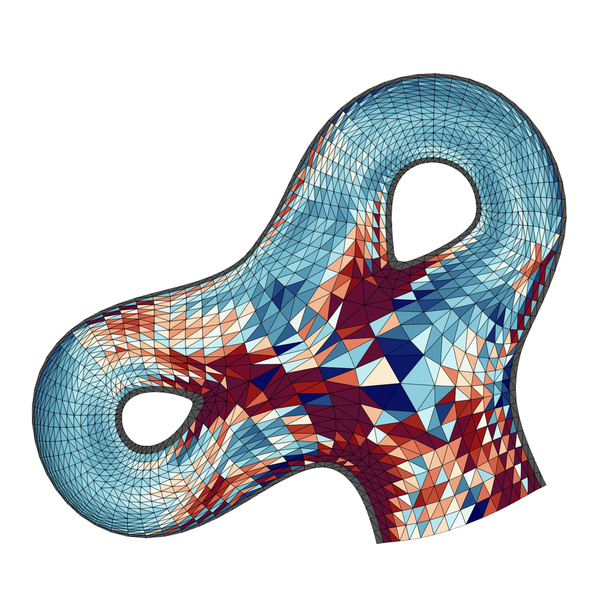}
	\includegraphics[width=0.24\textwidth]{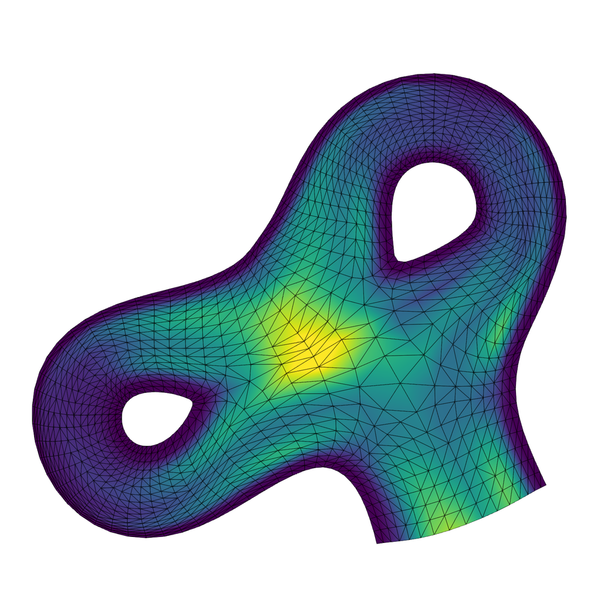}
	\includegraphics[width=0.05\textwidth]{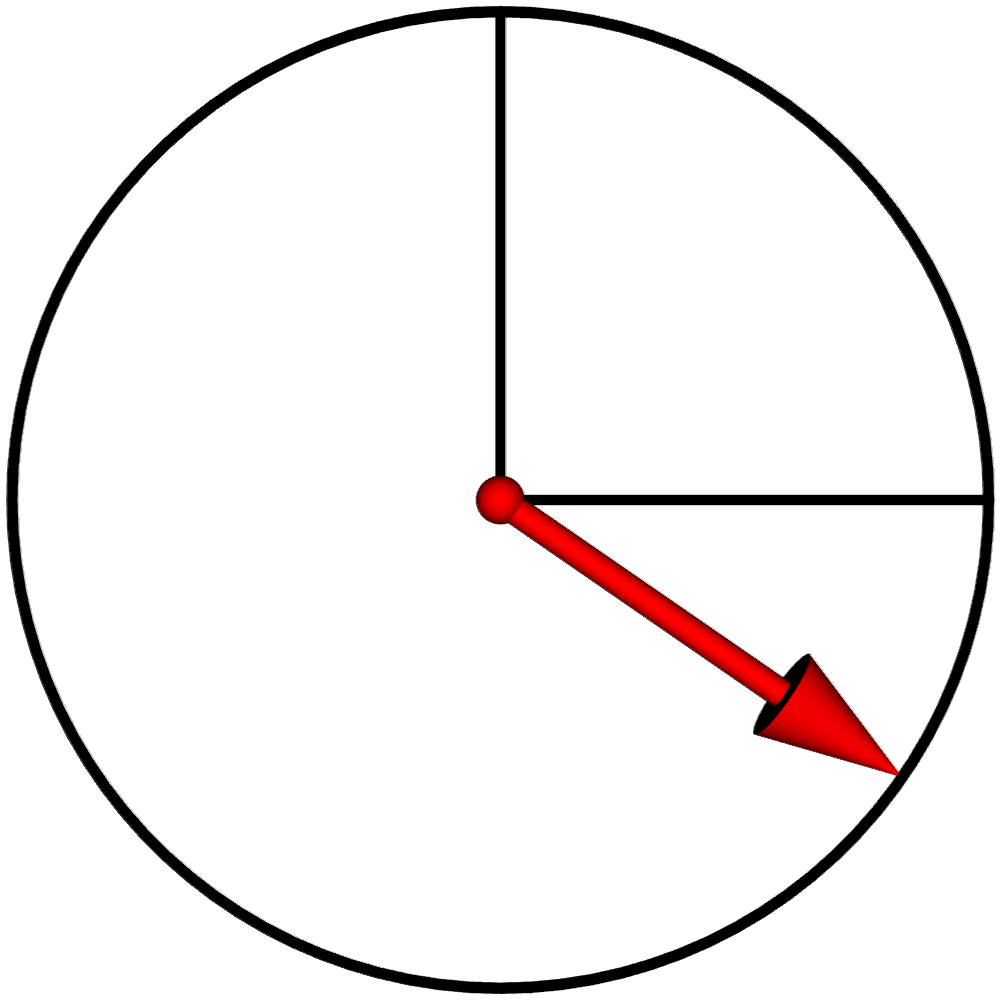}
	\caption{Results for two geometrically more complex examples. In both cases, we used tracking on the entire domain.
        On the left, we show the deformed configuration as a gray surface with the undeformed surfaces as a translucent overlay. 
        Furthermore, we visualize the direction of the force  leading to the maximal deformation in the cylinder. 
        In the middle, we see the resulting material distributions using the color map $0\hspace{1mm}$\protect\resizebox{.08\linewidth}{!}{\protect\includegraphics{figures/colorbar_cw.png}}$\hspace{1mm}0.2$. Boundary triangles for which all vertices are Dirichlet nodes 
        are shown in gray.
        On the right,  the magnitude of the deformation \(\y\) is displayed using the color map  $0\hspace{1mm}$\protect\resizebox{.08\linewidth}{!}{\protect\includegraphics{figures/colorbar_viridis.png}}$\hspace{1mm}0.7$.
        Additionally, on the far right, we show the 2D direction of the horizontal forces.
	}
	\label{fig:complex}
\end{figure}
%fffffffffffffffffffffffffffffffffffffffff

\section{Discussion}

The findings in this article draw a line from curved roof-type constructions via modeling and shape optimization of discrete thin shells to pessimistic formulations of bilevel stochastic programs. The challenge is that even in the deterministic case, it is well-known that standard compactness assumptions fail to ensure the existence of optimal solutions. \medskip

Assuming that the support of the underlying probability measure is compact, we have considered stochastic parameters and assessed the random upper-level outcome based on some (law-invariant) convex risk measure. 
For the pessimistic model, we have shown continuity of the resulting risk functional if the random perturbation admits a Lebesgue density, the set of potential forces is a polyhedron and the lower level goal function is real analytic. 
Alternatively, we have investigated a regularized model where the leader also hedges against lower level solutions that are close to optimality. 
The risk functionals emerging from this regularized problem are automatically lower semicontinuous. 
In both situations, the existence of optimal solutions can be guaranteed under a compactness condition.
We have developed a proof-of-concept numerical implementation that applies a pessimistic bilevel strategy to a mechanical optimal design problem, using a stochastic gradient descent approach to compute locally optimal solutions of the pessimistic model.

\medskip

In closing, we would like to point out several possible directions for future research.
In the numerical optimization, it would be interesting to consider interior point methods to solve the ``original'' leader's and follower's problem which incorporate hard constraints instead of the regularization used here. 
From the point of view of elasticity, it would be interesting to study the nonlinear model \eqref{eq:free_energy} instead of its linearized equivalent and investigate the associated nonuniqueness issue in the lower level problem. 
In fact, this would lead to a proper trilevel problem and bring new challenges for theoretical and numerical investigations. 
Furthermore, it would be worthwhile to investigate the infinite-dimensional variational problem of thin shell or volume elasticity with appropriate function spaces and from the perspective of optimization with continuous PDE constraints.
In the present paper, ``risky'' decisions can be penalized in the objective function, but  the leader ensures that the perturbed material parameters are feasible regardless of the realization, via a restriction of the design variable to the set $\mathcal{U}_\Upsilon$. 
Models, where this robust constraint is replaced with a system of chance or stochastic dominance constraints, can be expected to produce less conservative solutions, which improve the values of the leader's cost functional at the cost of some residual risk.

\section*{Acknowledgements}
We thank Kai Echelmeyer for providing an early version of a single level shape optimization algorithm
for discrete thin shells.
Furthermore, we thank Etienne Vouga for providing the meshes used throughout this work via \url{https://github.com/evouga/SelfSupporting}.

This work was partially supported by the Deutsche Forschungsgemeinschaft (DFG, German Research Foundation) {\sl via} project 211504053 - Collaborative Research Center 1060, 
project 390685813 - Hausdorff Center for Mathematics 
and the Collaborative Research Center TRR 154.

\bibliographystyle{spmpsci} 
\bibliography{Submission}

\end{document}